\documentclass{scrartcl}

\usepackage[T1]{fontenc}
\usepackage[utf8]{inputenc}
\usepackage[english]{babel}
\usepackage{lmodern}
\usepackage{exscale}
\usepackage[babel]{microtype}
\usepackage{mathtools, mathrsfs}
\usepackage{amssymb}
\usepackage{tikz}
\usetikzlibrary{matrix}

\usepackage{csquotes}
\usepackage[style=alphabetic, sorting=nyt, maxnames=10, maxalphanames=5]{biblatex}
\DeclareDelimFormat[textcite]{multinamedelim}{\textendash}
\DeclareDelimAlias[textcite]{finalnamedelim}{multinamedelim}

\usepackage{todonotes}
\DeclareDocumentCommand{\najib}{som}{\todo[color=blue!20, \IfBooleanT{#1}{inline}, \IfValueT{#2}{#2}]{N: #3}}
\DeclareDocumentCommand{\erik}{som}{\todo[color=green!20, \IfBooleanT{#1}{inline}, \IfValueT{#2}{#2}]{E: #3}}

\usepackage[numbered]{bookmark}
\usepackage{hyperref}
\hypersetup{
  colorlinks,
  urlcolor = {blue!80!black},
  linkcolor = {red!70!black},
  citecolor = {green!50!black}
}

\usepackage{amsthm}

\newtheorem{theorem}{Theorem}[subsection]
\newtheorem{proposition}[theorem]{Proposition}
\newtheorem{corollary}[theorem]{Corollary}
\newtheorem{lemma}[theorem]{Lemma}

\newtheorem{thmx}{Theorem}

\theoremstyle{definition}
\newtheorem{definition}[theorem]{Definition}
\newtheorem{notation}[theorem]{Notation}

\theoremstyle{remark}
\newtheorem{example}[theorem]{Example}
\newtheorem{remark}[theorem]{Remark}

\NewDocumentCommand{\bcirc}{}{\mathbin{\bar{\circ}}}
\NewDocumentCommand{\hotimes}{}{\mathbin{\hat{\otimes}}}

\NewDocumentCommand{\bZ}{}{\mathbb{Z}}
\NewDocumentCommand{\bN}{}{\mathbb{N}}
\NewDocumentCommand{\bQ}{}{\mathbb{Q}}

\NewDocumentCommand{\sM}{}{\mathsf{M}}
\NewDocumentCommand{\sN}{}{\mathsf{N}}
\NewDocumentCommand{\sP}{}{\mathsf{P}}
\NewDocumentCommand{\sQ}{}{\mathsf{Q}}

\NewDocumentCommand{\sI}{}{\mathsf{I}}
\NewDocumentCommand{\Com}{}{\mathsf{Com}}
\NewDocumentCommand{\Lie}{}{\mathsf{Lie}}
\NewDocumentCommand{\Ass}{}{\mathsf{Ass}}
\NewDocumentCommand{\Pois}{}{\mathsf{Pois}}
\NewDocumentCommand{\sE}{}{\mathsf{E}}

\NewDocumentCommand{\cH}{}{\mathcal{H}}
\NewDocumentCommand{\cQ}{}{\mathcal{Q}}
\NewDocumentCommand{\cF}{}{\mathcal{F}}
\NewDocumentCommand{\cP}{}{\mathcal{P}}

\NewDocumentCommand{\SymG}{}{\mathfrak{S}}
\NewDocumentCommand{\Mod}{}{\mathrm{Mod}}
\NewDocumentCommand{\ModK}{}{\Mod_{\Bbbk}}
\NewDocumentCommand{\gModK}{}{\mathrm{g}\ModK}

\NewDocumentCommand{\triv}{}{\mathrm{triv}}
\NewDocumentCommand{\sgn}{}{\mathrm{sgn}}

\DeclareMathOperator{\ch}{ch}
\DeclareMathOperator{\chtwo}{ch}
\DeclareMathOperator{\diag}{diag}
\DeclareMathOperator{\Ind}{Ind}
\DeclareMathOperator{\Res}{Res}
\DeclareMathOperator{\Rep}{\mathcal{R}}
\DeclareMathOperator{\End}{End}
\DeclareMathOperator{\Sat}{\mathcal{S}}
\DeclareMathOperator{\Aut}{Aut}
\DeclareMathOperator{\IA}{\mathsf{IA}}
\DeclareMathOperator{\GL}{GL}
\DeclareMathOperator{\Hom}{Hom}
\DeclareMathOperator{\tr}{tr}
\DeclareMathOperator*{\colim}{colim}

\title{Plethysm for characters of relative operads and props}
\author{Najib Idrissi\thanks{Université Paris Cité, Sorbonne Université, CNRS, IMJ-PRG, F-75013 Paris, France. Email: \href{mailto:najib.idrissi-kaitouni@u-paris.fr}{najib.idrissi-kaitouni@u-paris.fr}} \and Erik Lindell\thanks{Institut for Matematiske Fag, Københavns Universitet, Universitetsparken 5, 2100
København Ø, Denmark. Email: \href{mailto:erikjlindell@gmail.com}{erikjlindell@gmail.com}}}
\date{September 2026}

\begin{document}

\maketitle

\begin{abstract}
  We investigate the relationship between symmetric functions and the representation theory of operads, relative operads, and props.
  We extend the classical character map for symmetric sequences to relative bisymmetric sequences and symmetric bimodules.
  We introduce new operations on symmetric functions, the relative plethysm and the (connected) box product, which model via the character map the composition product of relative operads and the box product of prop(erad)s.
  As applications, we include the computation of characters for stable twisted cohomology of automorphism groups of free groups and the Albanese cohomology of the $\IA$-automorphism group.
\end{abstract}

\tableofcontents

\section{Introduction}
An operad is an object that consists of abstract multilinear operations, with some number of inputs and one output, which can be composed together in a natural way.
Each operad has an underlying symmetric sequence, i.e.\ sequences of representations of symmetric groups.
It is well-known that, in characteristic zero, the ring $\Rep(\SymG_*)$ of isomorphism classes of symmetric sequences is isomorphic to the completion of the ring $\Lambda$ of symmetric functions, i.e., symmetric polynomials in infinitely many variables.
The isomorphism, which can be viewed as a kind of decategorification, is realized by the character map $\ch:\Rep(\SymG_*) \to \hat{\Lambda}$, where $\hat{\Lambda}$ is the completion of $\Lambda$ with respect to the grading by monomial degree.

An operad is a monoid in the category of symmetric sequences under a certain monoidal structure, called the \emph{composition product}.
Under the character map, the composition product on the source corresponds to an operation called the \emph{plethysm} of symmetric functions.
The plethysm is easy to characterize in terms of explicit generators of $\Lambda$, making it practical for calculations.
In Section~\ref{sec:symmetric-functions-operads} below, we review the setup and proof of this correspondence in more detail.
The results of this paper concern similar correspondences for some generalizations of operads, namely \emph{relative operads}, \emph{props}, and properads.

\paragraph{Relative operads}

The first generalizations of operads that we consider are relative operads, also called ``Swiss-Cheese type operads.''
They encode operations with inputs of two kinds, or ``colors'', and one output.
Their underlying objects are relative bisymmetric sequences, that is, pairs $(\bar{\sN},\sN)$, where $\sN$ is a symmetric sequence,
and $\bar{\sN}$ is a bisymmetric sequence, i.e.\ a family of $(\SymG_m\times\SymG_n)$-modules,
which we think of as spaces with $m$ inputs of one color and $n$ of a different color.
Just like regular operads, relative operads are monoids for a certain composition product:
\begin{equation*}
  (\bar{\sM},\sM)\bcirc(\bar{\sN},\sN) \coloneqq \bigl(\bar{\sM}\bcirc(\bar{\sN},\sN),\sM\circ \sN\bigr).
\end{equation*}

The decategorification of relative bisymmetric sequences is the ring of relative bisymmetric functions, which is the product ring $\hat{\Lambda}_{x,y}\times\hat{\Lambda}_x$, where $\hat{\Lambda}_{x,y}$ is the completion of the ring $\Lambda_{x,y}$ of bisymmetric functions, i.e.\ symmetric polynomials in two sets of variables $x$ and $y$, and $\hat{\Lambda}_x$ is the completion of the ring $\Lambda_x$ of symmetric functions in the $x$-variables.
The character map
\[
  \chtwo:\Rep(\SymG_*\times\SymG_*)\to \hat{\Lambda}_{x,y}
\]
from the ring of isomorphism classes of bisymmetric sequences is an isomorphism.
In Section~\ref{sec:rel-op} below, we define a ``relative plethysm''
\[\bcirc:\hat{\Lambda}_{x,y}\times(\hat{\Lambda}_{x,y}\times\hat{\Lambda}_x)\to \hat{\Lambda}_{x,y}.\]
Similarly to the classical plethysm, the relative plethysm is easily characterized in terms of explicit generators and therefore practical for calculations.
We prove that it is the character of the relative composition product:

\begin{thmx}[Theorem~\ref{thm:relative-composition-plethysm}]
  Given a bisymmetric sequence $\bar{\sM}$ and a relative bisymmetric sequence $(\bar{\sN}, \sN)$, over a field of characteristic zero, we have:
  \begin{equation*}
    \chtwo\bigl(\bar{\sM} \bcirc (\bar{\sN}, \sN)\bigr) = \chtwo(\bar{\sM}) \bcirc \bigl( \chtwo(\bar{\sN}), \ch(\sN)\bigr).
  \end{equation*}
\end{thmx}

\paragraph{Props and properads}

A different generalization of operads is \emph{props} (sometimes written ``PROPs''), which encode multilinear operations with several inputs and outputs.
Props were initially introduced by \textcite{MacLane1965}, though we take here a point of view closer to the operadic one~\cite{Vallette2007}.
Properads are a better-behaved version of props, as they are amenable to methods such as Koszul duality.
The underlying objects of props are symmetric bimodules, that is families $(\sM(m,n))_{m,n\in\bN}$ such that $\sM(m,n)$, the space of operations with $n$ inputs and $m$ outputs, is endowed with a left $\SymG_m$-action and right $\SymG_n$-action.
Under the character map, the ring of symmetric bimodules is isomorphic to $\hat{\Lambda}_{x,y}$, the completed ring of bisymmetric functions.

In the first part of Section~\ref{sec:propic-char}, we deal with \emph{saturated} symmetric bimodules, that is, free monoids under horizontal composition. We introduce a saturation operation on bisymmetric functions (Definition~\ref{def:sat-char}), and we prove that it corresponds to the saturation of symmetric bimodules under the character map.

Props are (almost) monoids under the \emph{box product}, see~\cite{Vallette2007} or~\cite[Définition~4.1.1.3]{LerayPhD}.
In the second part of Section~\ref{sec:propic-char}, we introduce a box product (Definition~\ref{def:box-prod}):
\[\boxtimes:\hat{\Lambda}_{x,y}\times\hat{\Lambda}_{x,y}\to\hat{\Lambda}_{x,y}.\]
The construction of the box product was inspired by a similar construction for modular operads appearing in \cite{GetzlerKapranov1998}.
The box product of bisymmetric functions is more complicated to calculate than the plethysm and relative plethysm, but can still be expressed in relatively simple terms on generators and is thus also useful for calculations.
Moreover, we introduce a connected version of the box product, $\boxtimes_c$ (Definition~\ref{def:conn-box-prod}), which is suitable for properads and defines an actual monoidal structure on symmetric bimodules.
We introduce a corresponding operation on bisymmetric functions (Definition~\ref{def:conn-box-prod-char}):
\begin{equation}
  \boxtimes_c : \hat{\Lambda}_{x,y}\times\hat{\Lambda}_{x,y}\to\hat{\Lambda}_{x,y}.
\end{equation}

\begin{thmx}[Theorem~\ref{thm:box-product-character}, Theorem~\ref{thm:conn-box-product-character}]
  Let $\sM$, $\sN$ be symmetric bimodules.
  The character of their (connected) box product satisfies, over a field of characteristic zero:
  \begin{align*}
    \chtwo(\sM \boxtimes \sN) & = \chtwo(\sM) \boxtimes \chtwo(\sN), & \chtwo(\sM \boxtimes_c \sN) & = \chtwo(\sM) \boxtimes_c \chtwo(\sN).
  \end{align*}
\end{thmx}

\paragraph{Applications}

The plethysm and its generalizations are not only interesting from the perspective of being a ``decategorification'' of interesting operations on symmetric sequences, relative symmetric sequences and symmetric bimodules, but they are also useful in practice.
Many interesting objects appearing naturally in algebra, geometry and topology are equipped with operadic or prop-structures and using symmetric functions is thereby a tractable way to understand their representation theory.
For example, \textcite{GaroufalidisGetzler2017} used this correspondence to study twisted cohomology of mapping class groups and the Torelli Lie algebra. A similar method was used by \textcite[Section 6]{KupersRandal-Williams2020} to study the character of stable cohomology groups of Torelli groups into irreducible representations of symplectic groups,
by relating this decomposition to the character of a certain symmetric sequence of stable twisted cohomology of the mapping class group. In the final section of this paper, we use our generalizations of the plethysm to obtain similar results.

The second author~\cite{Lindell2022-AutFn} computed the stable rational cohomology groups of $\Aut(F_n)$, where $F_n$ is the free group on $n$ generators,
with certain ``bivariant'' twisted coefficients. These cohomology groups naturally form a prop, whose structure was studied by \textcite{KawazumiVespa2023}.
Combining these results, one obtains a simple description of the prop in terms of two generators. As a symmetric bimodule,
it can be expressed as the saturation of a very simple symmetric bimodule, so in Section~\ref{sec:applications}, we apply our results to calculate its character. We also calculate
the character of the closely related prop associated to the (shift of) the operad $\Com$, which encodes commutative algebras. The structure of this prop was studied by \textcite{EHLVZ2024}.

In \cite{Lindell2024IA}, the second author applied the results of \cite{Lindell2022-AutFn} to study the rational cohomology of $\IA_n$, which is the kernel of the natural map
$\Aut(F_n)\to\GL_n(\bZ)$. In the appendix of \cite{Lindell2024IA}, Katada used the results of the paper to calculate the part of the stable cohomology generated by the first cohomology under the cup product, typically known as the ``Albanese cohomology''.
The stable Albanese cohomology had been previously determined up to degree 3, in work of \textcite{Kawazumi2005-MagnusExpansions}, \textcite{Pettet05} and \textcite{Katada2022}.
Similarly to in \cite{KupersRandal-Williams2020}, one of the results of \cite{Lindell2024IA} relates the decomposition of a certain symmetric sub-bimodule of the stable bivariant twisted cohomology of $\Aut(F_n)$
to the decomposition of the Albanese cohomology into irreducible representations of $\GL_n(\bQ)$. This symmetric bimodule may also be described as the saturation of a
symmetric bimodule whose character is easy to write down by hand, thereby allowing us to use our results to calculate the character of the stable Albanese cohomology in arbitrary degree and thereby determining its decomposition into irreducibles in an efficient way.
In Appendix \ref{sec:appendixA}, we include the character, computed by implementing our results in Mathematica, up to degree 5 (at which point it is already quite unwieldy to write down).

\paragraph{Conventions}

We let $\bN = \{0, 1, \dots\}$ be the set of nonnegative integers.
For $n \in \bN$, we write $\SymG_n$ for the $n$th symmetric group and $\SymG_\infty = \bigcup_{n \in \bN} \SymG_n$.
We fix a commutative ring $\Bbbk$.
We work in the categories $\ModK$ of $\Bbbk$-modules and $\gModK$ of (cohomologically) $\bZ$-graded $\Bbbk$-modules.
In most of the paper, we will be interested in the case where $\Bbbk$ is a field of characteristic zero (and especially in Sections~\ref{sec:char-map}, \ref{sec:char-map-bisym}, and~\ref{sec:char-sat-box}), but we will try to state our results in as much generality as possible.

For $\lambda:=(\lambda_1\ge\lambda_2\ge\cdots\ge\lambda_l\ge 0)$ a partition, we denote by $|\lambda|:=\lambda_1+\cdots+\lambda_l$
its weight and by $l(\lambda)$ its length, i.e.\ the largest $l$ such that $\lambda_l\neq 0$.
We sometimes write $\lambda \, \vdash \, |\lambda|$ to express that $\lambda$ is a partition of $n = |\lambda|$.

\paragraph{Acknowledgements}

The authors thank Frédéric Han, Dan Petersen, and Thomas Willwacher for useful discussions, and Vladimir Dotsenko, Johan Leray, and Robin Stoll for useful comments.
We also thank the anonymous referees for their careful reading and useful comments.

The authors acknowledge support from project ANR-22-CE40-0008 SHoCoS.
N.I.\ also acknowledges support from project ANR-20-CE40-0016 HighAGT
and contributes to the IdEx University of Paris ANR-18-IDEX-0001.
E.L. was supported by the Knut and Alice Wallenberg Foundation through grant
no. 2022.0278. He is also grateful to the Copenhagen Centre for Geometry and Topology for their hospitality during the writing of this paper.

\section{Symmetric functions and operads}\label{sec:symmetric-functions-operads}

In this section, we recall the basics of the theory of symmetric functions and plethysm, symmetric sequences and composition product, and the link between the two.
We refer to~\cite{Macdonald1995} or~\cite[Chap.~7]{Stanley1999} for symmetric functions and~\cite{LodayVallette2012} for operads.
Characters of (cyclic and modular) operads are studied in~\cite{GinzburgKapranov1994,GetzlerKapranov1998}.

\subsection{Symmetric functions}

For $k\ge 1$, consider the graded polynomial ring $\bZ[x_1,\dots,x_k]$, where each generator has degree one. The symmetric group $\SymG_k$ acts (on the left) on this graded ring by permuting the generators via $\sigma \cdot x_i = x_{\sigma(i)}$ for $\sigma \in \SymG_k$. A \emph{symmetric polynomial} in $k$ variables is an element of the invariant subring
\[ \Lambda_k:=\bZ[x_1,\dots,x_k]^{\SymG_k}. \]
We write $\Lambda_k^n$ for the degree $n$ part of this graded ring. There is a graded ring homomorphism $\rho_k:\Lambda_{k+1}\to\Lambda_k$, defined by $\rho_k(x_{k+1}) = 0$ and $\rho_k(x_i) = x_i$ for  $1 \leq i \leq k$.

\begin{definition}
  The \emph{ring of symmetric functions} is the graded ring
  \[ \Lambda:=\lim(\cdots\to\Lambda_{k+1}\xrightarrow{\rho_k}\Lambda_k\to\cdots\to \Lambda_1). \]
  We denote the degree $n$ part of $\Lambda$ by $\Lambda^n$ and let $\hat{\Lambda}$ denote the completion of $\Lambda$ with respect to the induced filtration.
  As graded $\bZ$-modules, $\Lambda$ and $\hat{\Lambda}$ decompose as:
  \[ \Lambda=\bigoplus_{n\ge 0}\Lambda^n \subsetneq \hat{\Lambda}=\prod_{n\ge 0}\Lambda^n. \]
\end{definition}
Since we will deal with \emph{graded} representations of symmetric groups, we will also consider the rings $\Lambda((\hbar))$ and $\hat{\Lambda}((\hbar))$ of formal Laurent series.
Note that these rings are bigraded, with $\Lambda^n$ in bidegree $(n,0)$ and $\hbar$ in bidegree $(0,1)$.

\begin{notation}\label{not:lambda-x-y}
  Since we will sometimes change the names of the variables used to define symmetric functions, we will call the previous ring $\Lambda_x$ in some places.
  The ring $\Lambda_y$ will be defined similarly except we replace each $x_i$ by $y_i$.
\end{notation}

\begin{notation}\label{not:f-alpha}
  Any symmetric function $f \in \hat{\Lambda}$ can be written uniquely as an infinite sum $f = \sum_\alpha f_\alpha x^\alpha$, where $\alpha = (\alpha_1, \alpha_2, \dots)$ is a multi-index, $\alpha_i \geq 0$ and all but finitely many $\alpha_i$ vanish, each $f_\alpha$ is in $\bZ$, and $x^\alpha = \prod_{i=1}^\infty x_i^{\alpha_i}$.
  Since $f$ is symmetric, we have $f_\alpha = f_{\sigma \cdot \alpha}$ for all $\alpha$ and $\sigma \in \SymG_\infty$.
  Moreover, $f$ belongs to $\Lambda \subset \hat{\Lambda}$ if and only if there exists $d \in \bN$ such that if $|\alpha| \coloneqq \sum_i \alpha_i$ is bigger than $d$, then $f_\alpha$ vanishes.
\end{notation}

Let us introduce some classical elements of $\Lambda$.
For $n \in \bN$, the $n$th \emph{elementary symmetric function} $e_n \in \Lambda^n$, the $n$th \emph{complete symmetric function} $h_n \in \Lambda^n$, and the $n$th \emph{power sum} $p_n \in \Lambda^n$ are given by:
\begin{align}
  e_n
  & \coloneqq \sum_{i_1 < \dots < i_n} x_{i_1} \dots x_{i_n};
  & h_n
  & \coloneqq \sum_{i_1 \leq \dots \leq i_n} x_{i_1} \dots x_{i_n};
  & p_n
  & \coloneqq \sum_i x_i^n = x_1^n + x_2^n + \dots \text{(for } n \geq 1\text{)}, \; p_0 \coloneqq 1.
\end{align}
For example, $e_0 = h_0 = p_0 = 1$, $e_1 = h_1 = p_1 = x_1 + x_2 + \dots$, and $p_2 = h_2 - e_2$.

Another classical family of elements of $\Lambda$ is given by Schur functions.
These admit several different definitions; let us give a combinatorial definition in terms of Young tableaux.

\begin{definition}
  Given a partition $\lambda$, a \emph{Young tableau} of shape $\lambda$ is obtained by
  filling the boxes of the Young diagram corresponding to $\lambda$ with non-negative integers
  (where the same number may appear several times, or not at all).
  We say that a Young tableau is \emph{semistandard} if its entries weakly increase along each row and increase strictly down each column; see Figure~\ref{fig:y-tab} for an example.
\end{definition}

\begin{figure}[htbp]
  \begin{minipage}[b]{.59\textwidth}
    \caption{A semistandard Young tableau corresponding to the partition $\lambda = (5,3,2)$.}
  \end{minipage}
  \begin{minipage}[b]{.39\textwidth}
    \centering
    \begin{tikzpicture}
      \matrix [matrix of math nodes, row sep=-0.4pt, column sep=-0.4pt, nodes={draw, minimum size=5mm}]
      {
        1 & 1 & 3 & 4 & 4 \\
        3 & 5 & 5         \\
        4 & 6             \\
      };
    \end{tikzpicture}
  \end{minipage}
  \label{fig:y-tab}
\end{figure}

\begin{definition}\label{def:schur}
  Given $n\in\bN$, $\lambda$ a partition of $n$ and $T$ a semistandard Young tableau of shape $\lambda$,
  we can associate to it a multi-index $\alpha_T=(\alpha_1,\alpha_2,\ldots)$ where $\alpha_i$ is the number of occurrences of $i$ in $T$. We define the \emph{Schur function} $s_\lambda\in\Lambda^n$ by
  \[ s_\lambda:=\sum_{T} x^{\alpha_T}, \]
  where the sum is taken over all semistandard Young tableaux of shape $\lambda$.
\end{definition}

\begin{example}\label{exa:schur-e-h}
  Schur functions can be expressed in terms of the $e_n$ and $h_n$ and conversely.
  For example, we have that $s_{(1^n)} = e_n$ and $s_{(n)} = h_n$.
\end{example}

\begin{proposition}\label{prop:mult-basis}
  The family $(s_\lambda)_{|\lambda|=n}$ is an additive basis for $\Lambda^n$.
  Both families $(e_n)_{n \geq 1}$ and $(h_n)_{n \geq 1}$ are sets of algebraically independent generators of the ring $\Lambda$.
  Moreover, $(p_n)_{n \geq 1}$ is a set of algebraically independent generators for the ring $\Lambda \otimes \bQ$.
\end{proposition}

Using this proposition, we can define an involution
\begin{equation}\label{eq:omega}
  \omega : \Lambda \to \Lambda, \quad e_n \mapsto h_n.
\end{equation}
We have $\omega(s_\lambda)=s_{\lambda^T}$, where $\lambda^T$ denotes the transpose of the partition $\lambda$, and $\omega(p_n)=(-1)^{n-1}p_n$.
This involution is useful in our applications below.

Let us now turn our attention to the plethysm.

\begin{definition}
  Given $f,g \in \Lambda_x \coloneqq \Lambda$, the (one-variable) plethysm $f \circ g \in \Lambda_x$ is defined as follows.
  Let $g = \sum_\alpha g_\alpha x^\alpha$ (see Notation \ref{not:f-alpha}) and suppose that $g_\alpha \geq 0$ for all $\alpha$.
  Introduce temporary variables $z_i$ (only well defined up to permutation) such that, for a placeholder variable $t$, one has:
  \begin{equation}
    \prod_i (1 + z_i t) = \prod_\alpha (1 + x^\alpha t)^{g_\alpha};
  \end{equation}
  Then the plethysm is given by:
  \begin{equation}
    (f \circ g)(x_1, x_2, \dots) \coloneqq f(z_1, z_2, \dots).
  \end{equation}
\end{definition}

Concretely, $f \circ g$ is obtained by substituting the monomials that appear in $g$ (with multiplicity) into $f(x_1, x_2, \dots)$.
To fully define the plethysm (i.e., when $g$ has negative coefficients) we use the following propositions.

\begin{proposition}
  The plethysm is an endomorphism of rings in the first variable.
  Moreover, it intertwines with the involution $\omega$, i.e., for $f \in \Lambda^m$ and $g \in \Lambda^n$,
  \begin{equation}\label{eq:omega-plethysm}
    \omega(f \circ g) =
    \begin{cases}
      f \circ \omega(g),         & \text{if } n \text{ is even}; \\
      \omega(f) \circ \omega(g), & \text{otherwise}.
    \end{cases}
  \end{equation}
\end{proposition}

\begin{proposition}\label{pro:pn-circ}
  For all $n \geq 0$, the map $g \mapsto p_n \circ g$ is an endomorphism of rings, and:
  \begin{align}
    p_n \circ g & = g(x_1^n, \dots) = g \circ p_n;
    & \forall m \geq 0, \, p_n \circ p_m = p_{mn}.
  \end{align}
\end{proposition}

\begin{example}\label{exa:comput-plethysm}
  Since $h_2(x) = x_1^2 + x_1x_2 + \dots + x_2^2 + \dots$, it follows that
  \begin{equation}
    (p_r \circ h_2)(x) = p_r(x_1^2, x_1x_2, \dots, x_2^2, \dots) = x_1^{2r} + x_1^r x_2^r + \dots + x_2^{2r} + \dots = h_2(x_1^r, x_2^r, \dots).
  \end{equation}
\end{example}

Since $(p_n)_{n \geq 0}$ is a ring basis of $\Lambda \otimes \bQ$, these properties allow us to define the plethysm $f \circ g$ even when $g$ has negative coefficients: one can write $g = \sum_\alpha g_\alpha p_\alpha$ and then define $f \circ g \coloneqq \sum_\alpha g_\alpha (f \circ p_\alpha)$.

Moreover, the plethysm extends to an operation $\hat{\Lambda} \times \hat{\Lambda}^{\geq 1} \to \hat{\Lambda}$ (i.e., if the second variable vanishes at $x=0$), by noting that if $f \in \Lambda^k$ and $g \in \Lambda^l$ then $f \circ g$ is in $\Lambda^{kl}$.

Finally, we can extend the plethysm to $\Lambda((\hbar))$ and $\hat{\Lambda}((\hbar))$ as follows: let $g(x,\hbar)=\sum_{k,\alpha}g_{\alpha,k} x^{\alpha}(-\hbar)^{k}$,
where we assume $g_{\alpha,k}\ge 0$ for all $\alpha$ and $k$. We again introduce temporary variables $z_{i,k}$ such that for a placeholder variable $t$, we have
\begin{equation}
  \prod_{i,k}(1+z_{i,k}t)=\prod_{\alpha,k}(1+x^\alpha t)^{g_{\alpha,k}}.
\end{equation}
We then define
\begin{equation}
  (f\circ g)(x,\hbar)=f(((-\hbar)^{k}z_{i,k})_{i,k},\hbar).
\end{equation}

It is clear that this extension of the plethysm still satisfies that $(-)\circ g$ is a ring morphism and that $p_k\circ(-)$ is linear. Thus our extension of the plethysm can be characterized by saying that it is $\hbar$-linear in the first variable and setting $f \circ (\hbar g) = \hbar^k(f \circ g)$ for all $f \in \Lambda^k$ and all $g \in \Lambda$.
Note that for $f \circ g$ to be well-defined, we need $g(0)$ to be concentrated in strictly positive $\hbar$-degree~\cite[§2.4.6]{BergstroemDiaconuPetersenWesterland2023}.

In general, computing the plethysm of arbitrary symmetric functions is not an easy task.
Let us note one last property:

\begin{proposition}
  Plethysm is associative: for $f,g,h \in \Lambda$, we have $(f \circ g) \circ h = f \circ (g \circ h)$.
\end{proposition}

\subsection{Symmetric sequences and operads}\label{subsec:sym-seq-and-operads}

Next, we turn our attention to the ring of symmetric sequences, which can be viewed as a categorification of the ring of symmetric functions.

\begin{definition}
  A \emph{symmetric sequence} is a family $\sM = \{\sM(n)\}_{n \in \bN}$ of non-negatively graded $\Bbbk$-modules such that each $\sM(n)$ is endowed with a right action of $\SymG_n$.
  Such a sequence is said to be of \emph{finite-type} if each $\sM(n)$ is finite-dimensional.
\end{definition}

The (direct) sum $\sM \oplus \sN$ of symmetric sequences $\sM,\sN$ is defined term-wise.
The tensor product $\sM \otimes \sN$ is defined as:
\begin{equation}
  (\sM \otimes \sN)(n) \coloneqq \bigoplus_{k+l=n} \Ind_{\SymG_k \times \SymG_l}^{\SymG_n} \sM(k) \otimes \sN(l).
\end{equation}

The set of isomorphism classes of finite-type symmetric sequences forms a monoid for $\oplus$.
We denote by $\Rep_\Bbbk(\SymG_*)$ (or simply $\Rep(\SymG_*)$ when $\Bbbk = \bZ$) the Grothendieck group of that monoid, which forms a ring for $\otimes$. In addition, this ring has an involution, given by
$(M(n))_{n\in\bN}\mapsto (M(n)\otimes\sgn_n)_{n \in \bN}$.

Just as the ring of symmetric functions, the ring $\Rep_{\Bbbk}(\SymG_*)$ has an additional product:

\begin{definition}
  The \emph{composition product} of symmetric sequences $\sM,\sN$ is:
  \begin{equation}
    \sM \circ \sN \coloneqq \bigoplus_{r \geq 0} (\sM(r) \otimes \sN^{\otimes r})_{\SymG_r},
  \end{equation}
  where $(-)_{\SymG_r}$ denotes coinvariants under the (right) action of the symmetric group.
  The unit $\sI$ for this operation is given by $\sI(1) = \Bbbk$ and $\sI(n) = 0$ for $n \neq 1$.
\end{definition}

More precisely, in a given arity $n$, the composition product is given by
\begin{equation}\label{eq:comp-prod}
  (\sM \circ \sN)(n) = \bigoplus_{r \geq 0} \Bigl( \bigoplus_{k_1 + \dots + k_r = n} \Ind_{\SymG_{k_1} \times \dots \times \SymG_{k_r}}^{\SymG_n} \sM(r) \otimes \sN(k_1) \otimes \dots \otimes \sN(k_r) \Bigr)_{\SymG_r},
\end{equation}
where the action of $\SymG_r$ on the inner sum is given by permuting the $r$ factors $\sN(k_i)$ and acting on $\sM(r)$, and the action of $\SymG_n$ is induced from the action of the subgroup $\SymG_{k_1} \times \dots \times \SymG_{k_r} \subseteq \SymG_n$, where $\SymG_{k_i}$ acts on $\sN(k_i)$.

A symmetric sequence $\sM$ induces a (polynomial) functor $\cF_{\sM} : \gModK \to \gModK$ given, for $V \in \gModK$, by:
\begin{equation}
  \cF_{\sM}(V) \coloneqq \bigoplus_{n \in \bN} \bigl( \sM(n) \otimes V^{\otimes n} \bigr)_{\SymG_n},
\end{equation}
We then have a natural isomorphism of endofunctors $\cF_{\sM \circ \sN} \cong \cF_\sM \circ \cF_\sN$.
An \emph{operad} is a monoid in the category of symmetric sequences endowed with the composition product; if $\sM$ is an operad, then $\cF_\sM$ is a monad.

\begin{example}\label{exa:end-x}
  The prototypical example of operad is the \emph{endomorphism operad} $\End_X$ of $X \in \gModK$, defined by $\End_X(n) \coloneqq \Hom(X^n, X)$.
  The monoid structure is given by composition of multivariable maps.
\end{example}

\subsection{Character map}
\label{sec:char-map}

In this section, we assume that $\Bbbk$ is a field of characteristic zero.
Finally, let us look at the link between symmetric functions and sequences. This link is provided by the \emph{character map}. To define it, we first make the following definition:

\begin{definition}\label{def:graded-trace}
  Let $V$ be a non-negatively graded $\Bbbk$-vector space and $T:V\to V$ a linear operator of degree 0. Its graded \emph{trace} is defined as
  \begin{equation}
    \tr(T) \coloneqq \sum_{d\ge 0}(-\hbar)^d \tilde{\tr}(T_d)\in \Bbbk((\hbar)),
  \end{equation}
  where $T_d:V_d\to V_d$ is the restriction of $T$ to degree $d$ and $\tilde{\tr}(T_d)$ is the usual trace of $T_d$.
\end{definition}

\begin{definition}
  Let $\sM$ be a finite-type symmetric sequence.
  The \emph{character} of $\sM$ is the symmetric function $\ch(\sM)\in\hat{\Lambda}((\hbar))\otimes\Bbbk$ whose projection to $\Lambda_r((\hbar))\otimes\Bbbk$, for $r\ge 1$, is defined as the trace of the linear map $\cF_{\sM}(\Bbbk^r)\to\cF_{\sM}(\Bbbk^r)$ (with $\Bbbk$ considered as concentrated in degree 0) induced by the diagonal matrix $\diag(x_1,\ldots,x_r)$ acting on $\Bbbk^r$.
\end{definition}

Given a representation $M$ of $\SymG_n$, we can view it as a symmetric sequence by setting $\sM(n) = M$ and $\sM(k) = 0$ for $k \neq n$.
The character of the representation $M$ is then defined as the character of the associated symmetric sequence.
Note that it is homogeneous of polynomial degree $n$ in the $x_i$, and thus belongs to the subspace $\Lambda^n((\hbar)) \otimes \Bbbk \subset \hat{\Lambda}((\hbar)) \otimes \Bbbk$.

\begin{example}\label{ex:char-of-triv-and-sign}
  Let $M = \triv_2$ be the trivial representation of $\SymG_2$.
  Then $\cF_M(\Bbbk^r) = (\Bbbk^r)^{\otimes 2} / \langle x \otimes y - y \otimes x \rangle$ is spanned by symmetric binary tensors, and $\diag(x_1, \dots, x_r)$ has eigenvalues $\{x_i x_j \mid 1 \leq i, j \leq r\}$ (all with multiplicity one), so that $\ch(\triv_2) = h_2$.
  This generalizes easily: the character of the trivial representation of $\SymG_n$ satisfies $\ch(\triv_n) = h_n$.
  On the other hand, if $\sgn_n$ is the sign representation of $\SymG_n$, then we get $\ch(\sgn_n) = e_n$.
\end{example}

\begin{example}
  Recall that the irreducible representations of $\SymG_n$ are indexed by partitions of weight $n$ (see \cite[Section~4.2]{FultonHarris91}). For a partition $\lambda$, we
  denote the corresponding irreducible representation, called a \emph{Specht module}, by $S^\lambda$. We have $\ch(S^\lambda)=s_\lambda$.
\end{example}

\begin{example}
  Recall the involution $\omega$ from \eqref{eq:omega}.
  Then for a representation $M$ of $\SymG_n$, one has $\omega(\ch(M)) = \ch(M \otimes \sgn_n)$.
\end{example}

\begin{theorem}[{\cite[(7.3)]{Macdonald1995}}]\label{th:iso-char}
  The character map is an isomorphism of rings $\Rep(\SymG_*) \cong \hat{\Lambda}((\hbar))$.
\end{theorem}

Note that not every symmetric function corresponds to a genuine representation.
For example, $p_2 = h_2 - e_2$ is the character of the virtual representation $\triv_2 \ominus \sgn_2$.

The character map not only preserves the ring structure, but also the additional structure we have defined.
By Example \ref{ex:char-of-triv-and-sign}, we see that under the character map, the involution $\omega$ corresponds precisely to the involution given by tensoring arity-wise by the sign representation, introduced above.
The relationship between plethysm and composition product is given by the following proposition.
Note that we must assume that $\sN(0)$ is concentrated in positive degree for the composition product $\sM \circ \sN$ to remain of finite-type, so that the character map is defined on it.

\begin{proposition}\label{prop:composition-plethysm}
  For finite-type symmetric sequences $\sM,\sN$ such that $\sN(0)$ is concentrated in positive degree, one has:
  \begin{equation}
    \ch(\sM \circ \sN) = \ch(\sM) \circ \ch(\sN).
  \end{equation}
\end{proposition}

\begin{proof}
  The idea of the proof in the non-graded case is sketched in \cite[Appendix A, Eq. (7.3)]{Macdonald1995}, so let us adapt this idea to the graded case. Note that the definition in~\cite{Macdonald1995} uses invariants rather than coinvariants for the polynomial functors, but as we work over a field of characteristic zero, the two are isomorphic.

  We show the identity by proving that it holds after projection to $\Lambda_r((\hbar))$, for any $r\ge 1$.
  For $\sM$ and $\sN$ symmetric sequences,
  let $\cF_{\sM},\cF_{\sN}:\ModK\to\ModK$ be the respective associated functors. By the linearity of the composition product and plethysm in the first variable, we can assume that $\sM$
  is concentrated in degree 0. Let $(x)$ denote the diagonal endomorphism of $\Bbbk^r$ with eigenvalues $x_1,\ldots,x_r$. Writing $\alpha=(\alpha_1,\ldots,\alpha_r)$ for a multi-index, we have by definition of the character map that
  \[ \ch(\sN)(x)=\sum_{\alpha,k}(-\hbar)^k d_{\alpha,k} x^\alpha, \]
  where $d_{\alpha,k}$ is the dimension of the degree $k$ part of the eigenspace of $\cF_{\sN}((x))$ corresponding to the eigenvalue $x^\alpha$. For each $k\ge 0$, let $s_k:=\dim\cF_{\sN_k}(\Bbbk^r)$ and define variables $y_{1,k},\ldots,y_{s_k,k}$, by
  \[ \prod_{i=1}^{s_k}(1+y_{i,k}t)=\prod_\alpha(1+x^\alpha t)^{d_{\alpha,k}}. \]
  By diagonalizing each $\cF_{\sN_k}((x))$, we can find invertible $\phi_k\in\Hom_{\Bbbk}(\Bbbk^{s_k},\cF_{\sN_k}(\Bbbk^r))$ such that $\phi_k^{-1}\circ \cF_{\sN_k}((x))\circ\phi_k=(y_k):\Bbbk^{s_k}\to\Bbbk^{s_k}$, where we by abuse of notation write $(y_k):=\diag(y_{1,k},\ldots,y_{s_k,k})$.
  Letting $\Bbbk^{\underline{s}}$ denote the graded vector space given by $\Bbbk^{s_k}$ in degree $k$, we define $\phi\in\Hom_{\Bbbk}(\Bbbk^{\underline{s}},\cF_{\sN}(\Bbbk^r))$ so that its degree $k$ part is $\phi_k$.
  This means that $\phi^{-1}\circ \cF_{\sN}((x))\circ\phi=(y)$, where we again, by abuse of notation, write $(y)$ for the linear map $\Bbbk^{\underline{s}}\to\Bbbk^{\underline{s}}$ is $(y_k)$ when restricted to degree $k$. By definition of the character map, we thus have
  \begin{align*}
    \ch(\sM\circ \sN)(x) & = \tr \bigl( \cF_{\sM}(\cF_{\sN}((x))) \bigr) \\
    & = \tr \bigl( \cF_{\sM}(\phi\circ(y)\circ\phi^{-1}) \bigr) \\
    & = \tr \bigl( \cF_{\sM}(\phi)\circ\cF_{\sM}((y))\circ\cF_{\sM}(\phi^{-1}) \bigr) \\
    & = \tr \bigl( \cF_{\sM}((y)) \bigr) \\
    & = \sum_{k\ge 0}(-\hbar)^{k} \tr \bigl( \cF_{\sM}((y))_k \bigr) \\
    & = \ch(\sM) \bigl( ( (-\hbar)^{k }y_{i,k})_{i,k} \bigr) \\
    & = \bigl( \ch(\sM)\circ\ch(\sN) \bigr)(x).
  \end{align*}
  The first equality is by definition of the composition product and the character map; the last one is by definition of the plethysm.
  All the other steps are straightforward except the second to last.
  To justify it, note that if we forget $\hbar$,
  \begin{align*}
    \ch(\sM) \bigl( (y_{i,j})_{i,j} \bigr)
    & = \tr \bigl( \cF_{\sM}((y_{i,j})) \curvearrowright \bigoplus_{n \geq 0} \sM(n) \otimes_{\SymG_n} (\bigoplus_{j} \Bbbk^{s_j})^{\otimes n} \bigr) \\
    & = \tr \Bigl( \cF_{\sM}((y_{i,j})) \curvearrowright \bigoplus_{n,k} \bigl( \bigoplus_{j_1+\dots+j_n=k} \sM(n) \otimes \Bbbk^{s_{j_1}} \otimes \dots \otimes \Bbbk^{s_{j_n}} \bigr)/ \SymG_n \Bigr) \\
    & = \sum_k \tr \Bigl( \cF_{\sM}((y_{i,j}))_k \curvearrowright \bigoplus_{n} \bigl( \bigoplus_{j_1+\dots+j_n=k} \sM(n) \otimes \Bbbk^{s_{j_1}} \otimes \dots \otimes \Bbbk^{s_{j_n}} \bigr) / \SymG_n \Bigr).
  \end{align*}
  To recover $\hbar$, note that if we substitute $y_{i,j}$ by $(-\hbar)^j y_{i,j}$ in the resulting polynomial, then we multiply each monomial by $(-\hbar)$ to the power $j_1 + \dots + j_n$.
  The $k$th summand is thus multiplied by $(-\hbar)^{k}$, which gives the desired result.
\end{proof}

\begin{example}\label{exa:char-little-disks}
  Let $d\ge 2$ be an integer, and let $\sE_d$ denote the little $d$-disks operad, which is the topological operad such that $\sE_d(r)$ consists of $r$ unit disks embedded in the unit disk by translations and positive rescalings with disjoint interiors (see Figure \ref{fig:littledisks} for an example of an element).
  The homology $H_*(\sE_d,\bQ)$ is isomorphic to the operad of (unital) Poisson $d$-algebras $u\Pois_d$, i.e., algebras equipped with a unital graded commutative product and a Lie bracket of degree $d-1$.
  As a symmetric sequence, we have that $u\Pois_d \cong u\Com\circ\Lie_d$, where $u\Com$ encodes unital commutative algebras, and $\Lie_d$ is the $(d-1)$-fold operadic suspension of the Lie operad.

  We have $u\Com(n) = \triv_n$ for $n \geq 0$, so $\ch(u\Com) = \sum_{n \geq 0} h_n = \exp(\sum_{n \geq 1} p_n/n)$.
  On the other hand, $\Lie_d(n) = \Lie(n) \otimes \sgn_n^{\otimes (d-1)}[(n-1)(d-1)]$.
  Thanks to~\cite[Remark~2.3.8]{BergstroemDiaconuPetersenWesterland2023} (see also Witt's formula~\cite[Corollary~5.3.5]{Perrin1997}), we have:
  \begin{equation}
    \ch(\Lie)
    = \sum_{n \geq 1} \frac{-\mu_n}{n} \log(1-p_n)
    = \sum_{n,k \geq 1} \frac{\mu_n}{n} \cdot \frac{p_n^k}{k}
    ,
  \end{equation}
  where $\mu_n$ is the Möbius function.
  We thus get (using $\omega(p_n) = (-1)^{n-1} p_n$):
  \begin{equation}
    \ch(\Lie_d)
    = (-1)^d \hbar^{1-d} \sum_{n \geq 1} \frac{\mu_n}{n} \log(1 + (-1)^d \hbar^{(d-1)n} p_n).
  \end{equation}
  Therefore:
  \begin{equation}
    \begin{aligned}
      \log\ch(u\Pois_d)
      & = \Bigl(\sum_{k \geq 1} \frac{p_k}{k} \Bigr) \circ \Bigl( (-1)^d \hbar^{1-d} \sum_{n \geq 1} \frac{\mu_n}{n} \log(1 + (-1)^d \hbar^{(d-1)n} p_n) \Bigr) \\
      & = (-1)^d \sum_{n,k \geq 1} \hbar^{(1-d)k} \frac{\mu_n}{n k} \log(1 + (-1)^d \hbar^{(d-1)nk} p_{n k})                                                    \\
      & = (-1)^d \sum_{m \geq 1} \sum_{n | m} \hbar^{(1-d)m/n} \frac{\mu_n}{m} \log(1 + (-1)^d \hbar^{(d-1)m} p_m).
    \end{aligned}
  \end{equation}
  If $d = 1$, using $\sum_{n | m} \mu_n = 0$ for $m > 1$, then we get $\ch(u\Pois_1) = 1/(1-p_1) = \sum_{n \geq 0} p_1^n = \sum_{n \geq 0} \ch(\bQ[\SymG_n])$, which is consistent with the isomorphism of symmetric sequences $u\Pois_1 \cong u\Ass$ (reinterpretation of the PBW theorem).
  If $d = 2$ and we set $\hbar = 1$, then we recover the computation of~\cite[Proposition~2.3.7]{BergstroemDiaconuPetersenWesterland2023} (witnessing that the Koszul complex of $\Com$ is acyclic).
\end{example}

\begin{figure}[htbp]
  \centering
  \includegraphics[scale=0.4]{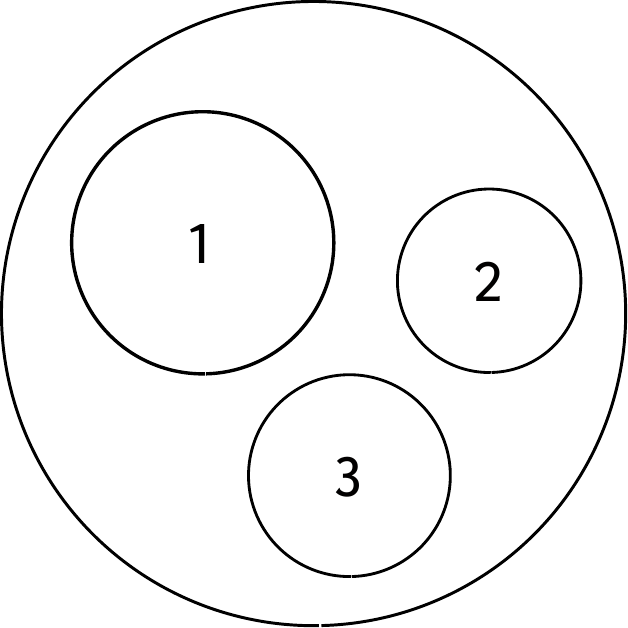}
  \caption{An element of $\sE_2(3)$.}
  \label{fig:littledisks}
\end{figure}

\section{Bisymmetric functions and relative operads}
\label{sec:bisym-operads}

We now introduce relative operads, a special kind of two-colored operads whose best-known representatives are the Swiss-Cheese operads.
We mimic the definitions of the previous sections to define a notion of relative plethysm, which models the composition product of relative operads.
We also compare our definition with a related notion introduced by \textcite{Koike1989}.

\subsection{Bisymmetric pairs and plethysm}
\label{sec:bisym-pairs}

\begin{definition}
  A \emph{bisymmetric function} is an element $\bar{f} = \bar{f}(x,y)$ in the ring $\Lambda_{x,y} \coloneqq \Lambda_x \otimes \Lambda_y$ (see Notation \ref{not:lambda-x-y}).
  We will also consider elements of the completed ring $\hat{\Lambda}_{x,y} \coloneqq \hat{\Lambda}_x \hotimes \hat{\Lambda}_y$ as well as in the rings of formal Laurent series $\Lambda_{x,y}((\hbar))$ and $\hat{\Lambda}_{x,y}((\hbar))$.
\end{definition}

\begin{notation}
  Given $f \in \Lambda$, we write $f(x) \coloneqq f \otimes 1 \in \Lambda_{x,y}$ and $f(y) \coloneqq 1 \otimes f \in \Lambda_{x,y}$.
\end{notation}

It follows from Proposition~\ref{prop:mult-basis} that $(e_n(x), \, e_n(y))_{n \geq 1}$ and $(h_n(x), \, h_n(y))_{n \geq 1}$ are each a set of algebraically independent generators of $\Lambda_{x,y}$,
that $(p_n(x), \, p_n(y))_{n \geq 1}$ is a set of algebraically independent generators of $\Lambda_{x,y} \otimes \bQ$ and that $(s_\lambda(x) \cdot s_\mu(y))_{\lambda,\mu}$ is an additive basis of $\Lambda_{x,y}$.
Note also that the involution $\omega:\Lambda\to\Lambda$ induces an involution $\omega\otimes\omega$ on $\Lambda_{x,y}$.

\begin{definition}
  A \emph{relative bisymmetric function} is a pair $(\bar{f}, f) \in \Lambda_{x,y} \times \Lambda_y$.
\end{definition}

\begin{definition}
  Given relative bisymmetric functions $(\bar{f}, f), (\bar{g}, g) \in \Lambda_{x,y} \times \Lambda_y$, the \emph{relative plethysm} is defined by:
  \begin{equation}
    (\bar{f}, f) \circ (\bar{g}, g) \coloneqq \bigl(\bar{f} \bcirc (\bar{g}, g), \; f \circ g\bigr).
  \end{equation}
  The symmetric function $f \circ g \in \Lambda_y$ is the usual plethysm of $f$ and $g$.
  The bisymmetric function $\bar{f} \bcirc (\bar{g}, g) \in \Lambda_{x,y}$ is defined as follows.
  Let $g = \sum_\beta g_\beta y^\beta$ and $\bar{g} = \sum_{\alpha, \beta} \bar{g}_{\alpha,\beta} x^\alpha y^\beta$ (Notation~\ref{not:f-alpha}) and assume that $\bar{g}_{\alpha,\beta}, g_\beta \geq 0$ for all $\alpha,\beta$.
  Introduce temporary variables $z_i, w_j$ such that, for a placeholder $t$,
  \begin{align}\label{eq:zi-wj}
    \prod_i (1 + z_i t) & = \prod_{\alpha,\beta} (1 + x^\alpha y^\beta t)^{\bar{g}_{\alpha,\beta}},
    &
    \prod_j (1 + w_j t) & = \prod_\beta (1+y^\beta t)^{g_\beta};
  \end{align}
  Then we have
  \begin{equation}
    (\bar{f} \bcirc (\bar{g}, g))(x_1, \dots; y_1, \dots) \coloneqq \bar{f}(z_1, \dots; w_1, \dots).
  \end{equation}
  As for the usual plethysm, this definition extends to functions with negative coefficients using the characterizations below, and also extends to the completion of the ring as well as power series.
\end{definition}

\begin{proposition}
  For any pair $(\bar{g}, g) \in \Lambda_{x,y} \times \Lambda_y$, the map $\bar{f} \mapsto \bar{f} \bcirc (\bar{g}, g)$ is an endomorphism of the ring $\Lambda_{x,y}$.
\end{proposition}
\begin{proof}
  This is immediate.
  For example, let us show that the map is additive.
  Given $\bar{f}, \bar{f}' \in \Lambda_{x,y}$ and $(\bar{g},g) \in \Lambda_{x,y} \times \Lambda_y$, with the notation of \eqref{eq:zi-wj}, we have
  \begin{align*}
    (\bar{f} + \bar{f}') \bcirc (\bar{g}, g)
    & = (\bar{f} + \bar{f}')(z_1, \dots; w_1, \dots) \\
    & = \bar{f}(z_1, \dots; w_1, \dots) + \bar{f}'(z_1, \dots; w_1, \dots) \\
    & = \bar{f} \bcirc (\bar{g}, g) + \bar{f}' \bcirc (\bar{g}, g).
    \qedhere
  \end{align*}
\end{proof}

Note that there are two involutions $\omega_x$, $\omega_y$ on $\Lambda_{x,y}$ analogous to $\omega$ from Equation~\eqref{eq:omega}: one that is defined by $\omega_x(h_n(x)) = e_n(x)$ and $\omega_x(h_n(y)) = h_n(y)$, and another that is defined similarly.
These behave with respect to $\bcirc$ similarly to Equation~\eqref{eq:omega-plethysm}.

\begin{lemma}
  Suppose that $\bar{f} \in \Lambda_{x,y}$ is of the form $\bar{f} = f(y)$ for some $f \in \Lambda_y$.
  Then for any pair $(\bar{g}, g)$, we have that $f(y) \bcirc (\bar{g}, g) = (f \circ g)(y)$.
\end{lemma}
\begin{proof}
  Since $\bar{f}$ does not depend on the variables $x_i$, $\bar{g}$ is irrelevant for computing $\bar{f} \bcirc (\bar{g},g)$.
  The definition of the $w_j$~\eqref{eq:zi-wj} matches those used to define $f \circ g$.
\end{proof}

Thus by Proposition~\ref{pro:pn-circ}:

\begin{corollary}
  For any $n \geq 0$, the map $(\bar{g}, g) \mapsto p_n(y) \bcirc (\bar{g}, g)$ is an endomorphism and:
  \begin{equation}\label{eq:pn-circ-y}
    p_n(y) \bcirc (\bar{g}, g) = g(y_1^n, y_2^n, \dots) = (g \circ p_n)(y).
  \end{equation}
  In particular, $p_k(y) \bcirc (\bar{g}, p_l) = p_{kl}(y)$ for all $k,l \geq 0$ and $\bar{g} \in \Lambda_{x,y}$.
\end{corollary}

On the other hand, if the first variable only depends on $x$, then we get:

\begin{proposition}
  For any $n \geq 0$, the map $(\bar{g}, g) \mapsto p_n(x) \bcirc (\bar{g}, g)$ is an endomorphism and:
  \begin{equation}\label{eq:pn-circ-x}
    p_n(x) \bcirc (\bar{g}, g) = \bar{g}(x_1^n, x_2^n, \dots; y_1^n, y_2^n, \dots) = \bar{g} \bcirc (p_n(x), p_n(y)).
  \end{equation}
  In particular, $p_k(x) \bcirc (p_l(x), g) = p_{kl}(x)$ and $p_k(x) \bcirc (p_l(y), g) = p_{kl}(y)$ for all $k,l \geq 0$ and $\bar{g} \in \Lambda_{x,y}$.
\end{proposition}
\begin{proof}
  Almost identical to the proof of Proposition~\ref{pro:pn-circ}, which itself follows immediately from the definition of the plethysm (see Example~\ref{exa:comput-plethysm} for $n=2$).
\end{proof}

Since the family $(p_n(x), p_n(y))_{n \geq 1}$ forms a set of algebraically independent generators for $\Lambda_{x,y}$, the above properties completely characterize the relative plethysm and allow us to define $\bar{f} \bcirc (\bar{g}, g)$ even when $\bar{g}$ or $g$ have negative coefficients, or when they are power series (with respect to $\hbar$), assuming that $(g,\bar{g})$ vanish at $0$.

Relative plethysm satisfies the following associativity property:
\begin{proposition}\label{prop:rel-plethysm-assoc}
  Let $\bar{f}, \bar{g}, \bar{h} \in \Lambda_{x,y}$ and $g,h \in \Lambda_y$.
  Then we have the equality:
  \begin{equation}
    \bigl( \bar{f} \bcirc (\bar{g}, g) \bigr) \bcirc (\bar{h}, h) = \bar{f} \bcirc \bigl( \bar{g} \bcirc (\bar{h}, h), \, g \circ h \bigr).
  \end{equation}
\end{proposition}
\begin{proof}
  Given the two previous results and the fact that $(p_n(x), p_n(y))_{n \geq 1}$ is a set of algebraically independent generators for $\Lambda_{x,y}$, it suffices to check the relation for $\bar{f} = p_n(x)$ or $p_n(y)$, $\bar{g} = p_m(x)$ or $p_m(y)$, and $g = p_k(y)$, which is straightforward.
  The result then follows from \eqref{eq:pn-circ-y} and \eqref{eq:pn-circ-x}.
\end{proof}

Let us now compare our definition of relative plethysm with a related notion introduced by \textcite{Koike1989} to study the characters of tensor products of representations.

\begin{definition}[{\cite[p.~79]{Koike1989}}]\label{def:koike-plethysm}
  Let $\bar{f}, \bar{g} \in \Lambda_{x,y}$ be bisymmetric functions and write $\bar{g} = \sum_{\alpha,\beta} g_{\alpha \beta} x^\alpha y^\beta$.
  Let $s_i, t_i$ be temporary variables such that, for a placeholder $z$,
  \begin{align}
    \prod_i (1+s_i z)
    & = \prod_{\alpha,\beta} (1+x^\alpha y^\beta z)^{g_{\alpha\beta}},
    & \prod_i (1+t_i z)
    & = \prod_{\alpha,\beta} (1+x^\beta y^\alpha z)^{g_{\alpha\beta}}.
  \end{align}
  Then the \emph{Koike plethysm} is given by:
  \begin{equation}
    (f \circ_K g)(x_1, \dots; y_1, \dots) \coloneqq f(s_1, \dots; t_1, \dots).
  \end{equation}
\end{definition}

The operation $\circ_K$ is clearly different from the operation $\bcirc$ that we introduced above.
While $\circ_K$ takes as input two bisymmetric functions, $\bcirc$ takes as input a bisymmetric function and a relative bisymmetric function.
Moreover, $\circ_K$ satisfies a kind of symmetry between $x$ and $y$, whereas our notion does not (and cannot, since it is supposed to model the relative composition product, for which the two colors play very different roles).
Nevertheless, we have the following equality:

\begin{proposition}
  If $\bar{f} = f(x)$ for some $f \in \Lambda_x$, then we have that
  \begin{equation}
    f(x) \bcirc (\bar{g}, 0) = f(x) \circ_K \bar{g}.
  \end{equation}
\end{proposition}

\subsection{Relative operads and relative composition product}
\label{sec:rel-op}

\begin{definition}
  A \emph{bisymmetric sequence} is a family $\bar{\sM} = \{ \bar{\sM}(m,n) \}_{m,n \in \bN}$ of graded $\Bbbk$-modules such that each $\bar{\sM}(m,n)$ is endowed with a right action of $\SymG_m \times \SymG_n$.
\end{definition}

Given bisymmetric sequences $\bar{\sM}, \bar{\sN}$, the direct sum $\bar{\sM} \oplus \bar{\sN}$ is defined term-wise.
The tensor product is defined by:
\begin{equation}
  (\bar{\sM} \otimes \bar{\sN})(m,n) \coloneqq \bigoplus_{k+l = m} \bigoplus_{k'+l'=n} \Ind_{\SymG_k \times \SymG_l \times \SymG_{k'} \times \SymG_{l'}}^{\SymG_m \times \SymG_n} \bar{\sM}(k,k') \otimes \bar{\sN}(l,l').
\end{equation}
We let $\Rep_{\Bbbk}(\SymG_* \times \SymG_*)$, or $\Rep(\SymG_* \times \SymG_*)$ when $\Bbbk = \bZ$, be the Grothendieck group of the monoid (for $\oplus$) of isomorphism classes of finite-type bisymmetric sequences, which forms a ring when endowed with $\otimes$.

\begin{definition}
  A \emph{relative bisymmetric sequence} is a pair $(\bar{\sM}, \sM)$ where $\sM$ is a symmetric sequence and $\bar{\sM}$ is a bisymmetric sequence.
\end{definition}

\begin{definition}\label{def:rel-op}
  Given relative bisymmetric sequences $(\bar{\sM}, \sM)$ and $(\bar{\sN}, \sN)$, the \emph{relative composition product} is defined by:
  \begin{equation}
    (\bar{\sM}, \sM) \circ (\bar{\sN}, \sN) \coloneqq \bigl(\bar{\sM} \bcirc (\bar{\sN}, \sN), \; \sM \circ \sN \bigr).
  \end{equation}
  The symmetric sequence $\sM \circ \sN$ is the usual composition product of $\sM$ and $\sN$.
  The bisymmetric sequence $\bar{\sM} \bcirc (\bar{\sN}, \sN)$ is (where $\sN$ is viewed as concentrated in bi-arity $(0,\_)$):
  \begin{equation}
    \bar{\sM} \bcirc (\bar{\sN}, \sN) \coloneqq \bigoplus_{m,n \in \bN} \bigl( \bar{\sM}(m,n) \otimes \bar{\sN}^{\otimes m} \otimes \sN^{\otimes n} \bigr)_{\SymG_m \times \SymG_n}.
  \end{equation}
\end{definition}

More precisely, the bisymmetric sequence $\bar{\sM} \bcirc (\bar{\sN}, \sN)$ is given, in bi-arity $(r,s)$, by:
\begin{multline}\label{eq:rel-comp-prod}
  (\bar{\sM} \bcirc (\bar{\sN}, \sN))(r,s) = \\ \bigoplus_{m,n \in \bN} \Bigl( \bigoplus_{\substack{\sum k_i = r\\\sum l_i + \sum l'_j = s}} \bar{\sM}(m,n) \otimes \Ind_{\prod_i \SymG_{k_i} \times (\prod_i \SymG_{l_i} \times \prod_j \SymG_{l'_j})}^{\SymG_r \times \SymG_s} \bigotimes_{i=1}^m \bar{\sN}(k_i, l_i) \otimes \bigotimes_{j=1}^{n} \sN(l'_j) \Bigr)_{\SymG_r \times \SymG_s}.
\end{multline}

Let $\bar{\sM}$ be a bisymmetric sequence.
The induced (polynomial) functor $\cF_{\bar{\sM}} : \ModK \times \ModK \to \ModK$ is given, for $V, W \in \ModK$, by:
\begin{equation}
  \cF_{\bar{\sM}}(V,W) \coloneqq
  \bigoplus_{m,n \in \bN} \bigl( \bar{\sM}(m,n) \otimes V^{\otimes m} \otimes W^{\otimes n} \bigr)_{\SymG_m \times \SymG_n}.
\end{equation}
In particular, a relative bisymmetric sequence $(\bar{\sM}, \sM)$ induces an endofunctor:
\begin{align}
  \cF_{\bar{\sM},\sM} & : \ModK \times \ModK \to \ModK \times \ModK,
  & (V,W)                                        & \mapsto \bigl( \cF_{\bar{\sM}}(V,W), \cF_{\sM}(W) \bigr).
\end{align}
We then have a natural isomorphism of endofunctors $\cF_{\bar{\sM}, \sM} \circ \cF_{\bar{\sN}, \sN} \cong \cF_{\bar{\sM} \bcirc (\bar{\sN}, \sN), \; \sM \circ \sN}$.

A \emph{relative operad} is a monoid in the category of relative bisymmetric sequences for the relative composition product.
If $(\bar{\sM}, \sM)$ is a relative operad, then  $\cF_{\bar{\sM}, \sM}$ is a monad.
The terminology comes from~\cite{Voronov1999}, who defined them as pairs $(\sQ, \sP)$ where $\sP$ is a plain operad and $\sQ$ is an operad in the category of right $\sP$-modules.
Such operads are sometimes called ``Swiss-Cheese type operads,'' owing to their original appearance.
They can also be defined as colored operads $\sQ$ with two colors, e.g., red and blue, such that operations with a red output may only have red inputs, and operations with blue output may have both red and blue inputs.

\begin{example}
  The prototypical example of a relative operad is the \emph{relative endomorphism operad} $(\End_{X,Y}, \End_Y)$ of a pair $(X,Y) \in \ModK \times \ModK$, where $\End_Y$ is from Example~\ref{exa:end-x} and $\End_{X,Y}(m,n) \coloneqq \Hom(X^m \times Y^n, X)$.
\end{example}

\subsection{Character map for bisymmetric sequences}
\label{sec:char-map-bisym}

In this section, we assume that $\Bbbk$ is a field of characteristic zero.
Let us now relate bisymmetric sequences and functions.

\begin{definition}
  Let $\bar{\sM}$ be a finite-type bisymmetric sequence.
  The \emph{character} of $\bar{\sM}$ is the bisymmetric function:
  \begin{equation}
    \chtwo(\bar{\sM}) \in \hat{\Lambda}_{x,y}((\hbar)) \otimes \Bbbk,
  \end{equation}
  whose projection onto $\Lambda_r \otimes \Lambda_s \otimes \Bbbk$ (which consists in polynomials in $r+s$ variables, symmetric in the first $r$ and symmetric in the last $s$), for $r,s \geq 1$, is the graded trace (see Definition~\ref{def:graded-trace}) of the linear map $\cF_{\bar{\sM}_d}(\Bbbk^r, \Bbbk^s) \to \cF_{\bar{\sM}_d}(\Bbbk^r, \Bbbk^s)$ induced by the pair of diagonal matrices $\bigl(\diag(x_1,\dots,x_r), \diag(y_1, \dots, y_s)\bigr)$.
\end{definition}

\begin{proposition}
  The character map is an isomorphism of rings between $\Rep(\SymG_* \times \SymG_*)$ and $\hat{\Lambda}_{x,y}((\hbar))$.
\end{proposition}
\begin{proof}
  This follows from Theorem~\ref{th:iso-char} and the easily checked fact that if a $(\SymG_m \times \SymG_n)$-representation $\bar{M}$ is of the form $M \otimes M'$, where $M$ is a representation of $\SymG_m$ and $M'$ is a representation of $\SymG_n$, then the character of $\bar{M}$ is the product of the characters of $M$ and $M'$, i.e., $\chtwo(\bar{M})(x,y) = \ch(M)(x) \cdot \ch(M')(y)$.
\end{proof}

In order to illustrate this notion, let us compute the following important example.

\begin{example}\label{ex:regular-representation}
  Let $G$ be a finite group and $R:=\bQ[G]$ its regular representation.
  Let $\{V_1,\ldots,V_k\}$ be the irreducible representations of $G$.
  Then by Maschke's theorem~\cite[Theorem~4.1.1]{EGHLSVY2009}, we have
  \begin{equation}
    \bQ[G] \cong \bigoplus_i \End(V_i) = \bigoplus_{i} V_i^* \otimes V_i.
  \end{equation}
  Note that the cited theorem only states that this is an isomorphism of left $G$-modules, but it is clear that the map is an isomorphism of bimodules.
  Moreover, $V_i^*$ viewed as a right module (by the inverse of the transpose) is isomorphic to $V_i$.
  For $G = \SymG_n$, if we apply the character map we thus get the following result:
  \begin{equation}\label{eq:regular-rep}
    \chtwo(\bQ[\SymG_n]) = R_n(x,y) \coloneqq \sum_{\lambda \,\vdash\, n} s_\lambda(x) s_\lambda(y).
  \end{equation}
\end{example}

We now get to the main result of this section.

\begin{theorem}\label{thm:relative-composition-plethysm}
  Given a bisymmetric sequence $\bar{\sM}$ and a relative bisymmetric sequence $(\bar{\sN}, \sN)$, we have that
  \begin{equation}
    \chtwo\bigl(\bar{\sM} \bcirc (\bar{\sN}, \sN)\bigr) = \chtwo(\bar{\sM}) \bcirc \bigl( \chtwo(\bar{\sN}), \ch(\sN)\bigr).
  \end{equation}
\end{theorem}

\begin{proof}
  The proof is similar to that of Proposition~\ref{prop:composition-plethysm}.
  We deal with the case where both sequences are concentrated in degree zero; the general case follows (just like how we pulled $(-\hbar)$ out of the polynomials in Proposition~\ref{prop:composition-plethysm}).
  We show that, for any $r,s\ge 1$, the identity holds when projecting to $\Lambda_r\otimes\Lambda_s$. Let $(x):\Bbbk^r\to\Bbbk^r$ and $(y):\Bbbk^s\to\Bbbk^s$ denote
  the diagonal endomorphisms with eigenvalues $x_1,\ldots,x_r$ and $y_1,\ldots,y_s$, respectively. For multi-indices $\alpha,\beta$ we then have
  \[ \chtwo(\bar{\sN})(x,y)=\sum_{\alpha,\beta}d_{\alpha,\beta}x^\alpha y^\beta, \]
  where $d_{\alpha,\beta}$ is the dimension of the eigenspace of $\cF_{\bar{\sN}}((x),(y)):\cF_{\bar{\sN}}(\Bbbk^r,\Bbbk^s)\to \cF_{\bar{\sN}}(\Bbbk^r,\Bbbk^s)$ corresponding to the eigenvalue $x^\alpha y^\beta$. As before, we also have that
  \[ \ch(\sN)(y)=\sum_{\beta}d_\beta y^\beta, \]
  where $d_\beta$ is the dimension of the eigenvalue of $\cF_{\sN}((y)):\cF_{\sN}(\Bbbk^s)\to\cF_{\sN}(\Bbbk^s)$ corresponding to the eigenvalue $y^\beta$. Letting $t:=\dim\cF_{\bar{\sN}(\Bbbk^r,\Bbbk^s)}$ and $u:=\dim\cF_{\sN}(\Bbbk^s)$, we define sets of variables $z_1,\ldots,z_t$ and $w_1,\ldots,w_u$ by
  \[ \prod_i(1+z_it) = \prod_{\alpha,\beta}(1+x^\alpha y^\beta t)^{d_{\alpha,\beta}},\quad \prod_j(1+w_jt)=\prod_{\beta}(1+y^\beta t)^{d_\beta}. \]

  Again, we can diagonalize $\cF_{\bar{\sN}}((x),(y))$ and find isomorphisms $\phi\in\Hom_{\Bbbk}(\Bbbk^t,\cF_{\bar{\sN}}(\Bbbk^r,\Bbbk^s))$ and $\psi\in\Hom_{\Bbbk}(\Bbbk^u,\cF_{\sN}(\Bbbk^s))$ such that
  $\phi^{-1}\circ \cF_{\bar{\sN}}((x),(y))\circ\phi=(z)$ and $\psi^{-1}\circ\cF_{\sN}((y))\circ\psi=(w)$. Recall that
  \[ \cF_{(\bar{\sM},0)}\circ\cF_{(\bar{\sN},\sN)}=\cF_{(\bar{\sM}\bcirc (\bar{\sN},\sN),0\circ\sN)}, \]
  which means in particular that $\cF_{\bar{\sM}\bcirc (\bar{\sN},\sN)}$ is the functor given on $(V,W)$ by
  \[ \cF_{\bar{\sM}}(\cF_{\bar{\sN}}(V,W),\cF_{\sN}(W)) \]
  and similarly on morphisms. Applying this, we get
  \begin{align*}
    \chtwo(\bar{\sM}\bcirc(\bar{\sN},\sN))(x,y)
    & =\tr(\cF_{\bar{\sM}\bcirc (\bar{\sN},\sN)}((x),(y))) \\
    & =\tr(\cF_{\bar{\sM}}(\cF_{\bar{\sN}}((x),(y)),\cF_{\sN}((y)))) \\
    & =\tr(\cF_{\bar{\sM}}(\phi\circ(z)\circ\phi^{-1},\psi\circ (w)\circ\psi^{-1})) \\
    & =\tr( \cF_{\bar{\sM}}(\phi,\psi)\circ\cF_{\bar{\sM}}((z),(w))\circ\cF_{\bar{\sM}}(\phi,\psi)^{-1}) \\
    & =\tr(\cF_{\bar{\sM}}((z),(w))) \\
    & =\chtwo(\bar{\sM})(z,w) \\
    & =\chtwo(\bar{\sM}) \bcirc (\chtwo(\bar{\sN}),\ch(\sN))(x,y),
  \end{align*}
  where we have used the definition of $\bcirc $ in the last step.
\end{proof}

\begin{remark}
  Given bisymmetric sequences $\bar{\sM}$, $\bar{\sN}$, the Koike plethysm (Definition~\ref{def:koike-plethysm}) of $\chtwo(\bar{\sM})$ and $\chtwo(\bar{\sN})$ is the character of the bisymmetric sequence given by:
  \begin{equation}
    \bar{\sM} \circ_K \bar{\sN} \coloneqq \bigoplus_{m,n \in \bN} \Bigl( \bar{\sM}(m,n) \otimes \bar{\sN}^{\otimes m} \otimes (\bar{\sN}^{\mathrm{op}})^{\otimes n} \Bigr)_{\SymG_m \times \SymG_n},
  \end{equation}
  where $\bar{\sN}^{\mathrm{op}}(m,n) \coloneqq \bar{\sN}(n,m)$.
\end{remark}

\begin{remark}
  We only deal with relative operads here as they are the ones that we need for our applications.
  However, it would be rather easy to adapt our results to the case of arbitrarily colored operads.
\end{remark}

\section{Bisymmetric functions and prop(erad)s}
\label{sec:propic-char}

A prop is an object which encodes operations that are allowed to have several inputs and several outputs, which can be composed in an arbitrary way.
Properads are a refinement of props, where operations can only be composed along connected graphs (see below).
The underlying objects of props and properads are symmetric bimodules, i.e., families of representations of the groups $\SymG_m^{\mathrm{op}}\times\SymG_n$.
Properads are monoids under a monoidal structure called the connected box product; props are (almost) monoids under the box product.
In this section, we construct operations on bisymmetric functions which decategorify the box product and the connected box product under the character map.
However, we must make a little detour through saturated bimodules first.
We first define the notion of saturation for bisymmetric functions and construct the (connected) box product of bisymmetric functions using this notion.

References for claims on props below include~\cite{Vallette2007}.

\subsection{Props}
\label{sec:props}

\begin{definition}
  A \emph{symmetric bimodule} is a family $\sM = \{ \sM(m,n) \}_{n,m \in \bN}$ of graded $\Bbbk$-modules such that each $\sM(m,n)$ is endowed with a left $\SymG_m$-action and a right $\SymG_n$-action. We call such a $\sM(m,n)$ a $(\SymG_m,\SymG_n)$-bimodule.
\end{definition}

Of course, this definition is equivalent to the notion of bisymmetric sequences from Section~\ref{sec:rel-op} since a left action can be turned into a right action and conversely.
However, we will use the term ``symmetric bimodule'' to emphasize a difference of points of view.
Bisymmetric sequences are the underlying objects of relative operads, whose operations have multiple inputs of two colors, and a single output.
On the other hand, symmetric bimodules are the underlying objects of props, whose operations have multiple inputs and multiple outputs, all of a single color.

Any symmetric bimodule has an associated bisymmetric sequence, which we will denote by $\sM$ as well.
That bisymmetric sequence consists of operations with two-colored inputs: the inputs of the first color correspond to outputs of the initial symmetric bimodule, while inputs of the second color correspond to inputs of the initial symmetric bimodule.
Given a symmetric bimodule $\sM$, we can therefore define its character $\chtwo(\sM)$ to be the character of the corresponding bisymmetric sequence.
Note that with our conventions, the $x$-variables of the character of a symmetric bimodule corresponds to the ``outputs'' of the symmetric bimodule, while the $y$-variables of the character correspond to the ``inputs''.

\begin{remark}
  A symmetric sequence $\sM$ can be seen as a symmetric bimodule by setting $\sM(1,n) = \sM(n)$ and $\sM(m,n) = 0$ for $m \neq 1$.
  We then have $\chtwo(\sM)(x,y) = p_1(x) \ch(\sM)(y)$.
\end{remark}

The analogue of the composition product of operads is the box product of props.

\begin{definition}[{\cite[Theorem~1]{Vallette2007}}]\label{def:box-prod}
  Let $\sM, \sN$ be symmetric bimodules.
  Their \emph{box product} $\sM \boxtimes \sN$ is the symmetric bimodule defined by:
  \begin{equation}
    (\sM \boxtimes \sN)(m,n) = \bigoplus_{N \geq 1} \Biggl( \bigoplus_{a,b \geq 1} \bigoplus_{\bar{k},\bar{l},\bar{\imath}, \bar{\jmath}} \Bbbk[\SymG_m] \otimes_{\SymG_{\bar{l}}} \sM(\bar{l}, \bar{k}) \otimes_{\SymG_{\bar{k}}} \Bbbk[\SymG_N] \otimes_{\SymG_{\bar{\jmath}}} \sN(\bar{\jmath}, \bar{\imath}) \otimes_{\SymG_{\bar{\imath}}} \Bbbk[\SymG_n] \Biggr) / {\sim},
  \end{equation}
  where the sum runs over $a$-tuples $\bar{k}, \bar{l}$ and $b$-tuples $\bar{\imath}, \bar{\jmath}$ such that $|\bar{l}| = m$, $|\bar{\imath}| = n$, and $|\bar{k}| = |\bar{\jmath}| = N$.
  We set $\sM(\bar{l}, \bar{k}) = \bigotimes_{p=1}^a \sM(l_p, k_p)$ and $\SymG_{\bar{l}} = \prod_{p=1}^a \SymG_{l_p} \subset \SymG_{|\bar{l}|}$ (and similarly for the others).
  The equivalence relation $\sim$ is generated, for $\theta \in \SymG_m$, $\omega \in \SymG_n$, $\sigma \in \SymG_N$, $\nu \in \SymG_a$, and $\tau \in \SymG_b$:
  \begin{equation}
    \theta \otimes \bigotimes_{q=1}^b m_q \otimes \sigma \otimes \bigotimes_{p=1}^a n_p \otimes \omega \sim
    \theta \tau_{\bar{l}}^{-1} \otimes \bigotimes_{q=1}^b m_{\tau^{-1}(q)} \otimes \tau_{\bar{k}} \sigma \nu_{\bar{\jmath}} \otimes \bigotimes_{p=1}^a n_{\nu(p)} \otimes \nu_{\bar{\imath}}^{-1} \omega,
  \end{equation}
  where $\tau_{\bar{l}}$, $\tau_{\bar{k}}$, $\nu_{\bar{\imath}}$, $\nu_{\bar{\jmath}}$ are the block permutations associated to $\tau,\nu$ and the partitions $\bar{k}, \bar{l}, \bar{\imath}, \bar{\jmath}$.
\end{definition}

While the previous definition is convenient to work with when we are interested in the representations of symmetric groups, it is not necessarily the most intuitive.
Just like the composition product of operads can be reformulated in terms of two-level trees, the box product of props can be reformulated in terms of two-level directed graphs, see Figure~\ref{fig:box-prod}.

\begin{figure}[htbp]
  \centering
  \begin{tikzpicture}[baseline={(0,-1)}]
    \draw [dashed] (-1,0) -- (2.5,0);
    \draw [dashed] (-1,-2) -- (2.5,-2);
    \node[draw, circle, fill=white] (x) at (0,0) {$x$};
    \node[draw, circle, fill=white] (y) at (1.5,0) {$y$};
    \node[draw, circle, fill=white] (x') at (0,-2) {$x'$};
    \node[draw, circle, fill=white] (y') at (1.5,-2) {$y'$};

    \draw[<-] (x) -- ++(-.75,0.75) node[above] {2};
    \draw[<-] (x) -- ++(-.25,0.75) node[above] {1};
    \draw[<-] (x) -- ++(.25,0.75) node[above] {5};
    \draw[<-] (x) -- ++(.75,0.75) node[above] {4};

    \draw[<-] (y) -- ++(-.25,0.75) node[above] {3};
    \draw[<-] (y) -- ++(.25,0.75) node[above] {6};

    \draw[->] (x') -- ++(-.25,-0.75) node[below] {2};
    \draw[->] (x') -- ++(.25,-0.75) node[below] {3};

    \draw[->] (y') -- ++(-.5,-0.75) node[below] {4};
    \draw[->] (y') -- ++(0,-0.75) node[below] {1};
    \draw[->] (y') -- ++(.5,-0.75) node[below] {5};

    \draw[->] (x) to[bend right] (x');
    \draw[->] (x) to (y');
    \draw[->] (y) to[bend left] (x');
    \draw[->] (y) to[bend right] (x');
    \draw[->] (y) to[bend left] (y');
  \end{tikzpicture}
  $\in (\sM \boxtimes \sN)(5,6)$.
  \caption{An element of the box product $\sM \boxtimes \sN$, where $x \in \sM(2, 4)$, $y \in \sM(3,2)$, $x' \in \sN(2,3)$, and $y' \in \sN(3,2)$.}
  \label{fig:box-prod}
\end{figure}

The box product is neither associative nor unital.
Consider the symmetric sequence $\sI$ viewed as a symmetric bimodule.
In general, $\sM \boxtimes \sI$ is not equal to $\sM$.
Moreover, if $\sM$ is a symmetric bimodule concentrated in arity $(2,2)$ with a single element, then $((\sI \boxtimes \sI) \boxtimes \sM)(2,2)$ is of dimension $2$, whereas $(\sI \boxtimes (\sI \boxtimes \sM))(2,2)$ is of dimension $1$~\cite[Section~C.4]{LerayPhD}.

A \emph{prop} is a symmetric bimodule $\sP$ equipped with a structure map $\sP \boxtimes \sP \to \sP$ satisfying several axioms (associativity, unitality), which we will not detail here.

\begin{remark}
  Note that these axioms are less straightforward than those of operads, as the box product is neither unital nor associative.
  An alternative definition of props is as symmetric bimodules $\sP$ which are monoids for the horizontal product equipped with a monoid structure map $\sP \boxtimes' \sP \to \sP$, where $\boxtimes'$ is the (symmetric monoidal) product defined like $\boxtimes$ except that graphs only have a single vertex on each level.
  We thank \textcite{StollPriv} for pointing this out to us and refer to~\cite[Définition~4.1.1.3]{LerayPhD} for a more detailed discussion.
\end{remark}

\begin{example}
  The prototypical example of a prop is the \emph{endomorphism prop} $\End_X$ where $\End_X(m,n) = \Hom(X^n, X^m)$.
\end{example}

\subsection{Properads}
\label{sec:properads}

Props are badly behaved for various reasons: the box product is neither associative nor unital, and Koszul duality does not exist for props.
\Textcite{Vallette2007} introduced the notion of \emph{properads}, which are a refinement of props that fix these issues; see also~\cite{HoffbeckLerayVallette2019} for a more recent account.
Properads are symmetric bimodules $\sP$ equipped with a structure map $\sP \boxtimes_c \sP \to \sP$, where the box product is replaced by the \emph{connected box product}.
Roughly speaking, the connected box product $\sM \boxtimes_c \sN$ is given by two-level \emph{connected} directed graphs, decorated by elements of the two symmetric bimodules $\sM$ and $\sN$.

\begin{definition}[{\cite[Proposition~1.5]{Vallette2007}}]\label{def:conn-box-prod}
  Let $\sM, \sN$ be symmetric bimodules.
  Their \emph{connected box product} $\sM \boxtimes_c \sN$ is the symmetric bimodule defined by:
  \begin{equation}
    (\sM \boxtimes_c \sN)(m,n) = \bigoplus_{N \geq 1} \Biggl( \bigoplus_{a,b \geq 1} \bigoplus_{\bar{k},\bar{l},\bar{\imath}, \bar{\jmath}} \Bbbk[\SymG_m] \otimes_{\SymG_{\bar{l}}} \sM(\bar{l}, \bar{k}) \otimes_{\SymG_{\bar{k}}} \Bbbk[\SymG^c_{\bar{k}, \bar{\jmath}}] \otimes_{\SymG_{\bar{\jmath}}} \sN(\bar{\jmath}, \bar{\imath}) \otimes_{\SymG_{\bar{\imath}}} \Bbbk[\SymG_n] \Biggr) / {\sim},
  \end{equation}
  where everything is exactly as in Definition~\ref{def:box-prod}.
  The only difference is that $\SymG^c_{\bar{k}, \bar{\jmath}}$ $\subset \SymG_N$ is the set of ``connected permutations,'' i.e., if we view $\bar{k}$ and $\bar{\jmath}$ as partitions of $N$, then $\sigma \in \SymG_N$ belongs to $\SymG^c_{\bar{k}, \bar{\jmath}}$ if and only if the graphs with vertices given by the blocks of the two partitions and edges given by $([i], [\sigma(i)])$ for $i \in \{1,\dots,N\}$ is connected.
\end{definition}

The connected box product is associative and unital, and the unit is given by $\sI$~\cite[Proposition~1.6]{Vallette2007}.
\emph{Properads} are symmetric bimodules which are monoids for the connected box product.
The relationship between the box product and the connected box product is given by the notion of saturation, which we now introduce.

\begin{definition}\label{def:saturation}
  Let $\sM$ be a symmetric bimodule.
  The \emph{saturation} of $\sM$ is the symmetric bimodule
  \begin{equation}
    \Sat(\sM) \coloneqq \bigoplus_{n \geq 1} (\sM^{\otimes n})_{\SymG_n},
  \end{equation}
  where the tensor product is given by horizontal concatenation (i.e., Day convolution, see e.g., \cite[Definition~1.6]{HoffbeckLerayVallette2019}).
\end{definition}

Note that if $\sM$ is \emph{output-reduced}
(i.e., $\sM(m,n) = 0$ if $m = 0$) then $\Sat(\sM) \cong \sI \boxtimes \sM $ and if it is \emph{input-reduced} (i.e., $\sM(m,n) = 0$ if $n = 0$) then similarly $\Sat(\sM)\cong \sM \boxtimes \sI$.
A symmetric bimodule is called \emph{saturated} if it is isomorphic to the saturation of some symmetric bimodule.

\begin{remark}
  Our definition differs from that of~\cite{Vallette2007}, who called a symmetric bimodule $\sM$ saturated if $\sM = \Sat(\sM)$.
  Such bimodules are necessarily trivial~\cite{StollPriv}.
\end{remark}

The saturation of the identity module, $\Sat(\sI)$, will play a special role.
It is given by:
\begin{equation}
  \Sat(\sI)(m,n) =
  \begin{cases}
    \Bbbk[\SymG_n], & \text{if } m = n; \\
    0,              & \text{otherwise}.
  \end{cases}
\end{equation}
Thanks to Example~\ref{ex:regular-representation}, we have that $\Sat(\sI)(n,n) \cong \bigoplus_{\lambda \, \vdash \, n} S^\lambda \otimes S^\lambda$ as a $(\SymG_n \times \SymG_n)$-representation.

The relationship between the box product and the connected box product is given by the following result.
\begin{proposition}[{\cite[Proposition~1.7]{Vallette2007}}]\label{prop:conn-box-prod-sat}
  Let $\sM, \sN$ be symmetric bimodules.
  Then the connected box product $\sM \boxtimes_c \sN$ and the box product $\sM \boxtimes \sN$ are related by:
  \begin{equation}
    \Sat(\sM \boxtimes_c \sN) = \sM \boxtimes \sN.
  \end{equation}
\end{proposition}

\subsection{Character of the saturation and the box product}
\label{sec:char-sat-box}

In this section, we assume that $\Bbbk$ is a field of characteristic zero.
The following proposition gives the relationship between the saturation of a symmetric bimodule and the relative composition product of Section~\ref{sec:rel-op}.
Let $\Com_c^0$ be the bisymmetric sequence given by $\Com_c^0(n,0) \coloneqq \triv_n$ for $n > 0$ and $\Com_c^0(n,m) = 0$ for $m \neq 0$ or $(m,n) = (0,0)$.
The following proposition is a reformulation of the description from~\cite[p.~4874]{Vallette2007}.

\begin{proposition}
  Let $\sM$ be a symmetric bimodule viewed as a bisymmetric sequence.
  Then $\Sat(\sM) = \Com_c^0 \bcirc (\sM, 0)$ as a bisymmetric sequence.
\end{proposition}

Note that the character of $\Com_c^0$ is given by:
\begin{equation}
  \chtwo(\Com_c^0) = \sum_{n \geq 1} h_n(x) = \exp \Bigl( \sum_{n \geq 1} \frac{p_n(x)}{n} \Bigr) - 1 \in \hat{\Lambda}_{x,y}.
\end{equation}

\begin{definition}\label{def:sat-char}
  The \emph{saturation} of a bisymmetric function $\bar{f}$ such that $\bar{f}(0,0)$ is concentrated in positive $\hbar$-degree is given by:
  \begin{equation}
    \Sat(\bar{f}) \coloneqq \sum_{n \geq 1} h_n(x) \bcirc (\bar{f}, 0).
  \end{equation}
\end{definition}

Note that $\Sat(\bar{f}) = \mathbf{exp}(\bar{f}) - 1$, where $\mathbf{exp}$ is a 2-variable version of the plethystic exponential~\cite[Definition~2.3.9]{BergstroemDiaconuPetersenWesterland2023}.

\begin{corollary}\label{cor:bichar-of-sat}
  The character of the saturation of a finite-type symmetric bimodule $\sM$ such that $\sM(0,0)$ is concentrated in positive degrees is given by:
  \begin{equation}
    \chtwo(\Sat(\sM)) = \Sat(\chtwo(\sM)).
  \end{equation}
\end{corollary}

\begin{example}
  We can use this result to compute the saturation of the symmetric bimodule $\sI$.
  Its character is simply given by $\chtwo(\sI) = h_1(x) h_1(y) = \sum_{i,j} x_i y_j$.
  We thus have that
  \begin{multline}\label{eq:sat-i}
    \chtwo(\Sat(\sI)) = \sum_{n \geq 1} h_n(x) \bcirc (h_1(x) h_1(y), 0) = \\ = \sum_{n \geq 1} \sum_{\substack{i_1 \leq \dots \leq i_n \\ j_1 \leq \dots \leq j_n}} x_{i_1} y_{j_1} \cdots x_{i_n} y_{j_n} = \prod_{i,j} (1 - x_i y_j)^{-1} - 1.
  \end{multline}
  Comparing this with \eqref{eq:regular-rep}, we recover~\cite[Chap.~I, Eq.~(4.3)]{Macdonald1995}.
\end{example}

Finally, let us provide a formula for the character of the box product of two symmetric bimodules.
Unfortunately, this formula is not as nice as the plethysms defined in the previous sections, but it is still useful computationally.

To define it, we recall that there is a scalar product on $\Lambda_{x,y}$, which is given by
\begin{equation}\label{eq:scalar-product}
  \langle p_{\lambda}(x)p_{\lambda'}(y),\,p_\mu(x)p_{\mu'}(y)\rangle \coloneqq \delta_{\lambda,\mu}\delta_{\lambda',\mu'} z_\lambda z_\mu,
\end{equation}
on the basis $\{p_\lambda(x)p_\mu(y)\}_{\lambda,\mu}$, where $z_\lambda \coloneqq \prod_i i^{m_i} m_i!$ and $m_i$ is the number of occurrences of $i$ in the partition $\lambda$ (and similarly for $z_\mu$).
If $T:\Lambda_{x,y}\to\Lambda_{x,y}$ is a linear map, we therefore have a well-defined notion of adjoint map $T^{\perp}:\Lambda_{x,y}\to\Lambda_{x,y}$. If $f\in\Lambda_{x,y}$, we denote by $f^\perp$ the adjoint
of the map given by multiplication by $f$.

\begin{example}
  Note that the scalar product and the adjoint are the canonical extensions of the corresponding notions on $\Lambda$.
  According to \cite[Example 5.3]{Macdonald1995}, if $f = f(p_1, p_2, \dots) \in \Lambda$ is expressed as a polynomial of power sums, then
  \begin{equation}\label{eq:adj-pn}
    p_n^\perp(f) = n \frac{\partial f}{\partial p_n}(p_1, p_2, \dots).
  \end{equation}
  For example, if $f = p_1^2 + p_2$, then $p_1^\perp(f) = 2p_1$ and $p_2^\perp(f) = 2$.
  These formulas allow computing the action of $p_n(x)^\perp$ and $p_n(y)^\perp$ on $\Lambda_{x,y}$ in a similar way.
  Since $f \mapsto f^\perp$ is a ring morphism, this completely characterizes adjoints.
\end{example}

Recall the bisymmetric function $R_n(x,y)$ from \eqref{eq:regular-rep}, \eqref{eq:sat-i}.
With this, we can define the box product of bisymmetric functions as follows:

\begin{definition}
  Let $\bar{f}(x,y), \bar{g}(x,y)$ be bisymmetric functions.
  Their \emph{box product} $\bar{f} \boxtimes \bar{g}$ is the bisymmetric function:
  \begin{equation}
    (\bar{f} \boxtimes \bar{g})(x,y) \coloneqq \Bigl( \sum_{n \geq 1} (R_n(x',y'))^{\perp} \bigl( \Sat(\bar{f})(x,y') \Sat(\bar{g})(x',y) \bigr) \Bigr)|_{x' = y' = 0}.
  \end{equation}
\end{definition}
The idea to use adjoints to model ``joining vertices by an edge'' is inspired by the proof of~\cite[Theorem~8.13]{GetzlerKapranov1998}.
The formula for the box product of bisymmetric functions is quite a bit more complicated than the formulas for the plethysm and relative plethysm above,
but is still computable. In particular, thanks to \cite[Equation (4.1)]{Macdonald1995}, we have
\begin{equation}
  R_n(x,y)^\perp = \sum_{\lambda \, \vdash \, n} z_\lambda^{-1} p_\lambda(x)^\perp p_\lambda(y)^\perp,
\end{equation}
where $z_\lambda$ is defined above and $p_\lambda(x)^\perp$, $p_\lambda(y)^\perp$ can be computed using Equation~\eqref{eq:adj-pn} and the multiplicativity of $f \mapsto f^\perp$.

\begin{theorem}\label{thm:box-product-character}
  Let $\sM$, $\sN$ be finite-type symmetric bimodules.
  The character of their box product satisfies:
  \begin{equation}
    \chtwo(\sM \boxtimes \sN) = \chtwo(\sM) \boxtimes \chtwo(\sN).
  \end{equation}
\end{theorem}

For the proof, we will need this representation theoretic lemma:

\begin{lemma}\label{lem:char-of-adjoint}
  Let $V$ be a $(\SymG_m,\SymG_n)$-bimodule and $W$ a $(\SymG_p,\SymG_q)$-bimodule, for $m\le p$ and $n\le q$. Then
  $$\ch(V)^\perp \ch(W)=\ch\left(\Hom_{\SymG_m\times\SymG_n}(V,\Res_{(\SymG_{m}\times\SymG_{p-m})\times(\SymG_{n}\times\SymG_{q-n})}^{\SymG_{p}\times\SymG_q}W)\right).$$
\end{lemma}

\begin{proof}
  This follows immediately from the analogous statement for one-variable symmetric functions, see~\cite[Proposition~8.10]{GetzlerKapranov1998} (where the adjoint $f^\perp$ is denoted $D(f)$).
\end{proof}

\begin{proof}[Proof of Theorem \ref{thm:box-product-character}]
  Let us start by introducing some auxiliary notation. If $\bar{f}\in\Lambda_{x,y}$, let us write $\bar{f}^{m,n}$ for the part of bidegree $(m,n)$, $\bar{f}^{m,-}$ for the part of degree $m$ in the $x$-variable and similarly $\bar{f}^{-,n}$ for the part of degree $n$ in the $y$-variable.
  If $\bar{f},\bar{g}\in\Lambda_{x,y}$, then in this notation, an equivalent way of writing the definition of their box product is
  \begin{align*}
    (\bar{f}\boxtimes\bar{g})(x,y)=\sum_{N\ge 1}R_N(x',y')^\perp\left(\Sat(\bar{f})^{-,N}(x,y')\Sat(\bar{g})^{N,-}(x',y)\right).
  \end{align*}
  We will use this alternative definition in the proof.

  Now note that if either $\sM(m,0)\neq 0$ or $\sN(0,n)\neq 0$, then these biarities do not contribute to the box product, so we can, without loss of generality, assume that $\sM$ is input-reduced and $\sN$ is output-reduced.
  This implies that $\Sat(\sM)=\sM\boxtimes\sI$ and $\Sat(\sN)=\sI\boxtimes\sN$, so since $\sI$ is concentrated in biarities $(n,n)$, we have
  \begin{align*}
    \Sat(\sM)(m,n) & =\left(\bigoplus_{a\ge 1}\bigoplus_{\bar{k},\bar{l}}\Bbbk[\SymG_m]\otimes_{\SymG_{\bar{l}}}\sM(\bar{l},\bar{k})\otimes_{\SymG_{\bar{k}}}\Bbbk[\SymG_n]\otimes\Bbbk[\SymG_n]\right)/\sim, \\
    & \cong\left(\bigoplus_{a\ge 1}\bigoplus_{\bar{k},\bar{l}}\Bbbk[\SymG_m]\otimes_{\SymG_{\bar{l}}}\sM(\bar{l},\bar{k})\otimes_{\SymG_{\bar{k}}}\Bbbk[\SymG_n]\right)/\sim'
  \end{align*}
  where the equivalence relation $\sim'$ is defined, for $\theta\in\SymG_m$, $\omega\in\SymG_n$ and $\tau\in\SymG_a$, by
  $$\theta\otimes m_1\otimes\cdots\otimes m_a\otimes \omega\sim' \theta\tau_{\bar{l}}^{-1}\otimes m_{\tau^{-1}(1)}\otimes\cdots\otimes m_{\tau^{-1}(a)}\otimes\tau_{\bar{k}}\omega.$$
  Similarly, we obtain
  \begin{align*}
    \Sat(\sN)(m,n)\cong\left(\bigoplus_{b\ge 1}\bigoplus_{\bar{\imath},\bar{\jmath}}\Bbbk[\SymG_m]\otimes_{\SymG_{\bar{\jmath}}}\sN(\bar{\jmath},\bar{\imath})\otimes_{\SymG_{\bar{\imath}}}\Bbbk[\SymG_n]\right)/\sim'.
  \end{align*}
  We see that for any $N\ge 1$, we have
  \begin{align*}
    & \Sat(\sM)(m,N)\otimes_{\SymG_N}\Bbbk[\SymG_N]\otimes_{\SymG_N}\Sat(\sN)(N,n) \\
    \cong & \left(\bigoplus_{a,b\ge 1}\bigoplus_{\bar{k},\bar{l},\bar{\imath},\bar{\jmath}}\Bbbk[\SymG_m]\otimes_{\SymG_{\bar{l}}}\sM(\bar{l},\bar{k})\otimes_{\SymG_{\bar{k}}}\Bbbk[\SymG_N]\otimes_{\SymG_{\bar{\jmath}}}\sN(\bar{\jmath},\bar{\imath})\otimes_{\SymG_{\bar{\imath}}}\Bbbk[\SymG_n]\right)/\sim'' \\
  \end{align*}
  where $\sim''$ is the equivalence relation generated by the equivalence relations in the two tensor factors. It is easily verified that this is the same equivalence relation as in the definition of the box product, so we obtain
  \begin{align*}
    (\sM\boxtimes\sN)(m,n) & \cong\bigoplus_{N\ge 1}\Sat(\sM)(m,N)\otimes_{\SymG_N}\Bbbk[\SymG_N]\otimes_{\SymG_N}\Sat(\sN)(N,n) \\
    & \cong\bigoplus_{N\ge 1}\Hom_{\SymG_N\times\SymG_N}\left(\Bbbk[\SymG_N],\Sat(\sM)(m,N)\otimes\Sat(\sN)(N,n)\right)
  \end{align*}
  where in the second step we used the self-duality of the regular representation $\Bbbk[\SymG_N]$. Applying Lemma \ref{lem:char-of-adjoint} and Corollary \ref{cor:bichar-of-sat}, we obtain
  \begin{align*}
    \chtwo(\sM\boxtimes\sN) & =\sum_{N\ge 1}R_N(x',y')^\perp\left( \chtwo\left(\Sat(\sM)\right)^{-,N}(x,y')\chtwo\left(\Sat(\sN)\right)^{N,-}(x',y)\right) \\
    & =\chtwo(\sM)\boxtimes\chtwo(\sN).\qedhere
  \end{align*}
\end{proof}

Finally, thanks to Proposition~\ref{prop:conn-box-prod-sat}, we can handle the connected box product as well.

\begin{definition}
  Define two elements $E, L$ of $\Lambda_{x}$ by:
  \begin{align}
    E & \coloneqq \sum_{r \geq 0} h_r(x) = \exp \bigl( \sum_{r \geq 1} \frac{p_r(x)}{r} \bigr),
    & L & \coloneqq \sum_{k \geq 1} \frac{\mu_k}{k} \log(1 + p_k(x)).
  \end{align}
\end{definition}
These two elements satisfy the relation $(E-1) \circ L = h_1 = L \circ (E-1)$~\cite[Proposition~2.3.7]{BergstroemDiaconuPetersenWesterland2023}.
Note that for $f \in \Lambda_{x,y}$, we have $(E(x) - 1) \bcirc (f, 0) = \Sat(f)$ (Definition~\ref{def:sat-char}).

\begin{definition}\label{def:conn-box-prod-char}
  Let $\bar{f}, \bar{g} \in \Lambda_{x,y}$ be bisymmetric functions.
  Their \emph{connected box product} $\bar{f} \boxtimes_c \bar{g}$ is the bisymmetric function:
  \begin{equation}
    \bar{f} \boxtimes_c \bar{g} \coloneqq L \bcirc (\bar{f} \boxtimes \bar{g}, 0).
  \end{equation}
\end{definition}

The following theorem is the analogue of Theorem~\ref{thm:box-product-character} for the connected box product and follows immediately from Proposition~\ref{prop:conn-box-prod-sat}, Theorem~\ref{thm:box-product-character}, the definition of the connected box product, and associativity of $\bcirc$ (Proposition~\ref{prop:rel-plethysm-assoc}).
\begin{theorem}\label{thm:conn-box-product-character}
  Let $\sM$, $\sN$ be finite-type symmetric bimodules.
  The character of their connected box product satisfies:
  \begin{equation}
    \chtwo(\sM \boxtimes_c \sN) = \chtwo(\sM) \boxtimes_c \chtwo(\sN).
  \end{equation}
\end{theorem}

\section{Applications}\label{sec:applications}

\subsection{Stable twisted cohomology of automorphism groups of free groups} For $n\ge 1$, let $F_n$ denote the free group on $n$ generators and let $H(n):=H_1(F_n,\bQ)$. For any $p,q\ge 0$, we consider the
$\Aut(F_n)$-representation
\[ B_n(q,p):=\Hom_{\bQ}(H(n)^{\otimes q},H(n)^{\otimes p}). \]
These representations assemble into a prop, inducing a prop-structure on the collection of cohomology groups
\begin{equation}
  \cH_n(q,p):=H^*(\Aut(F_n), B_n(q,p)).
\end{equation}
\begin{remark}
  As shown in~\cite{KawazumiVespa2023}, these cohomology groups actually have more structure: they form a so-called \emph{wheeled} prop,
  where the wheeled structure is induced by the duality pairing map $B_n(1,1)\to\bQ$.
\end{remark}

These cohomology groups have been studied by the second author~\cite{Lindell2022-AutFn} and several others (see for example \cite{DjamentVespa15,Djament19,RW18}). However, the only part of the cohomology which is well-understood is the \emph{stable} part, as we now recall.
There is a group homomorphism
\[ s_n:\Aut(F_n)\to\Aut(F_{n+1}) \]
given by extending
automorphisms to act trivially on the new generator. Furthermore, the standard inclusion $F_n\hookrightarrow F_{n+1}$ and projection $F_{n+1}\cong F_n*\bZ\to F_n$ induce an $\Aut(F_n)$-equivariant map
\[ \sigma_n:B_{n+1}(q,p)\to B_{n}(q,p), \]
where the source is considered an $\Aut(F_n)$-representation via the map $s_n$.
We thus get an induced map $(s_n,\sigma_n)^*$ in cohomology, and it follows from \cite{RW-Wahl17} that for $n$ sufficiently large in comparison to the degree, $p$ and $q$, this
map is an isomorphism. For a given $n$, the stable part of the cohomology is thus the cohomology in those degrees which lie in the stable range. Furthermore, we define the \emph{stable} cohomology of $\Aut(F_n)$ with the coefficients $B_n(p,q)$ as the limit
\begin{equation}
  H^*(\Aut(F_\infty),B_\infty(q,p)) \coloneqq \lim(\cdots\to \cH_{n+1}(q,p) \xrightarrow{(s_n,\sigma_n)^*} \cH_n(q,p)\to\cdots \to \cH_1(q,p)).
\end{equation}
Let us write $\cH = \cH_\infty$ for the (wheeled) prop formed by the stable cohomology groups. It turns out that it has a remarkably simple description,
obtained by combining the main result of \cite{Lindell2022-AutFn} with \cite[Theorem 4]{KawazumiVespa2023}:

\begin{theorem}
  The wheeled prop $\cH$ is isomorphic to the wheeled prop associated to $\Com[1]$, the shift of the commutative operad.
\end{theorem}

\begin{remark}
  The binary operation generating $\Com[1]$ corresponds to a class typically denoted $h_1\in \cH(1,2)$, which was introduced by \textcite{Kawazumi2005-MagnusExpansions}.
\end{remark}

A natural question to ask is how, for any $p,q\ge 0$, $\cH(q,p)$ decomposes into irreducible representations of $\SymG_q\times\SymG_p$.

\begin{theorem}[{\cite[Theorem A]{Lindell2022-AutFn}}]
  For $p,q \geq 0$, let $\sP(q,p)$ be the set of partitions of $\{1,\ldots,p\}$ with at least $q$ parts, $q$ of which are labeled $1,\ldots,q$ and the remaining parts unlabeled, with the natural (right) action of $\SymG_p \times \SymG_q$.
  Moreover, let $\cP(q,p) = \bQ\{\sP(q,p)\}[p-q]$ be the symmetric bimodule defined by this set of partitions.
  Then $\cH \cong \omega(\cP)$ where $\omega$ is the sign involution.
\end{theorem}
Consider the ungraded symmetric bimodule $\cQ$, defined by:
\begin{align}
  \cQ(q,p)
  & =
  \begin{cases}
    \triv_p & \text{ if }q \in \{0,1\},\ p\ge 1, \\
    0       & \text{ otherwise.}
  \end{cases}
  & \chtwo(\cQ)=\sum_{p\ge 1}(h_p(y)+ h_p(y)h_1(x)).
\end{align}
Then $\cP(q,p) = \bigl(\Sat(\cQ)\bigr)(q,p)[p-q]$, so that we get:

\begin{lemma}\label{lem:H}
  The wheeled prop $\cH$ satisfies $\cH(q,p) \cong \omega\bigl(\Sat(\cQ)(q,p)[p-q]\bigr)$.
\end{lemma}

Let us define a ``regrading'' morphism $\Psi$ from $\hat{\Lambda}_{x,y}$ to its ring of Laurent series by:
\begin{equation}\label{eq:map-Psi}
  \begin{aligned}
    \Psi : \hat{\Lambda}_{x,y} & \to \hat{\Lambda}_{x,y}((\hbar)), \\
    h_q(x)                     & \mapsto (-\hbar)^{-q} h_q(x),     \\
    h_p(y)                     & \mapsto (-\hbar)^{p} h_p(y),
  \end{aligned}
\end{equation}
or in other words,
\begin{equation}
  \Psi(f(x_1,\dots; y_1,\dots)) = f(-\hbar^{-1} x_1,\dots; -\hbar y_1,\dots) = f \bcirc (-\hbar^{-1} h_1(x), -\hbar h_1(y)).
\end{equation}
Then we have by Corollary \ref{cor:bichar-of-sat} that:
\begin{proposition}\label{prop:char-of-H}
  The character of $\cH$ is given by the involution of the regrading of the saturation of $\cQ$:
  \begin{equation}
    \chtwo(\cH) = \omega\Psi(\chtwo(\Sat(\cQ))) = \omega\Psi\Biggl(\sum_{n\ge 1} h_n(x) \bcirc \Bigl( \sum_{p\ge 1} \bigl(h_p(y) + h_p(y)h_1(x) \bigr) \Bigr) \Biggr).
  \end{equation}
\end{proposition}

The sub-prop of $\cH$ generated by the binary operation $h_1\in\cH(1,2)$ has been studied by \textcite{EHLVZ2024}.
Let us denote it by $\tilde{\cH}$. If we let $\tilde{\sP}$ denote the symmetric sub-bimodule of $\sP$ generated by partitions with no unlabeled parts, and $\cP = \{ \sP(q,p)[p-q] \}$ its graded version, then we have $\tilde{\cH} = \omega(\tilde{\cP})$.
Let us also define $\tilde{\cQ}$ by:
\begin{align}
  \tilde{\cQ}(q,p) & =
  \begin{cases}
    \triv_p & \text{ if }q=1,\ p\ge 1, \\
    0       & \text{ otherwise}.
  \end{cases}
  & \chtwo(\tilde{\cQ})                   & = \sum_{p\ge 1} h_p(y)h_1(x).
\end{align}
We have again that $\tilde{\cP}$ is a regraded version of $\Sat(\tilde{\cQ})$, so by Corollary~\ref{cor:bichar-of-sat}, we have
\begin{equation}
  \chtwo(\tilde{\cH}) = \omega\Psi(\chtwo(\Sat(\tilde{\cQ})))=\omega\Psi\Biggl(\sum_{n\ge 1}h_n(x)\bcirc\Bigl(\sum_{p\ge 1} h_p(y)h_1(x)\Bigr)\Biggr).
\end{equation}
Note that this sum is infinite even in a fixed degree; one has to consider a fixed pair (arity, degree) to get a finite sum.
The result of the computation of $\chtwo(\cH)$ in low degree and low arity is included in Section~\ref{sec:appendixA}.

\subsection{Stable algebraic cohomology of the \texorpdfstring{$\IA$}{IA}-automorphism group}
The $\IA$-automorphism group of $F_n$, which we
denote by $\IA_n$, is the kernel of the action of $\Aut(F_n)$ on the abelianization $H_1(F_n,\bZ)$. The stable rational cohomology of this group has been studied by the second author~\cite{Lindell2024IA}
and in a similar way by \textcite{HabiroKatada2023}. The homomorphism $s_n:\Aut(F_n)\to\Aut(F_{n+1})$ restricts to a homomorphism
\[ s_n:\IA_n\to\IA_{n+1}, \]
which we denote by the same symbol, for simplicity. As above, we may define
$H^*(\IA_\infty,\bQ)$. However, it should be noted that the cohomology does not stabilize in the same sense as above, i.e.\ $s_n^*$ is \emph{not} an isomorphism in any range $n\gg *$.
Instead, the short exact sequence
\[ 1\to\IA_n\to\Aut(F_n)\to\GL_n(\bZ)\to 1, \]
induces an action by $\GL_n(\bZ)$ on the cohomology $H^*(\IA_n,\bQ)$ and it is conjectured that with respect to this action, the cohomology stabilizes
in the sense of representation stability~\cite[Section 6.2]{ChurchFarb2013}. The first piece of evidence for this is that
\[ H^1(\IA_n,\bQ)\cong \Hom_\bQ(\Lambda^2 H(n),H(n)), \]
as representations of $\GL_n(\bZ)$, which was proven by \textcite{Kawazumi2005-MagnusExpansions}. In particular, $H^1(\IA_n,\bQ)$ is an \emph{algebraic} $\GL_n(\bZ)$-representation.
As part of the conjecture of representation stability, it is conjectured that for $n\gg *$, $H^*(\IA_n,\bQ)$ is an algebraic representation.

A second step in trying to understand the stable cohomology groups is to study
\[ H^*_A(\IA_n,\bQ):=\mathrm{Im}(\Lambda^* H^1(\IA_n,\bQ) \xrightarrow{\smile} H^*(\IA_n,\bQ)), \]
which is called the \emph{Albanese cohomology} of $\IA_n$. By the result of Kawazumi, this is a quotient of the exterior algebra on $\Hom_\bQ(\Lambda^2 H(n),H(n))$
and in particular an algebraic representation. Katada~\cite{Lindell2024IA} proved that for $n\gg *$, we have
\[ H^*_A(\IA_n,\bQ)\cong\Lambda^*\left(\Hom_\bQ(\Lambda^2 H(n),H(n))\right)/I, \]
where $I$ is an explicit quadratic ideal (as the specific description will not be relevant here, we refer to \cite{Lindell2024IA} for details). Below, we
apply the results of this paper to decompose this graded $\GL_n(\bQ)$-representation into irreducibles.
Such a decomposition can be obtained in degree 1 using the result of \textcite{Kawazumi2005-MagnusExpansions} above, and the decompositions were computed by \textcite{Pettet05} in degree 2 and by \textcite{Katada2022} in degree 3.

A next step in trying to understand $H^*(\IA_n,\bQ)$ is to study its \emph{algebraic} part. For any $\GL_n(\bZ)$-representation $V$, we define its algebraic part as
\[ V^{\mathrm{alg}} \coloneqq \Bigl( \colim_{\substack{W\subseteq V\\ W\text{ algebraic}}} W \Bigr) \subseteq V. \]
We may also define $H^*(\IA_\infty,\bQ)^{\mathrm{alg}}$ in a similar way as above.
If, for each $k \geq 0$, there exists a range $n \gg k$ such that $H^k(\IA_n,\bQ)$ is finite dimensional (or satisfies the weaker condition called ``Borel vanishing''), then we have~\cite{Lindell2024IA}:
\begin{equation}\label{eq:alg-cohomology-of-IA}H^k(\IA_\infty,\bQ)^{\mathrm{alg}} \cong (\bQ[y_4,y_8,\ldots]\otimes H_A^*(\IA_\infty,\bQ))^k,
\end{equation}
where $y_i$ is a class of degree $i$ which is invariant under the action of $\GL_\infty(\bZ):=\mathrm{colim}_{n}\GL_n(\bZ)$ (see also \cite{HabiroKatada2023} for a similar, but slightly weaker, result). In particular, this tells us that if the cohomology is stably finite dimensional and we want to decompose $H^*(\IA_\infty,\bQ)^{\mathrm{alg}}$ into irreducible representations,
we only need to decompose $H^*_A(\IA_\infty,\bQ)$. This is thus the goal of this section.

To see how the results of this paper are related, let us recall that the irreducible representations of $\GL_n(\bQ)$ are indexed by bipartitions, i.e.\ pairs $(\lambda,\mu)$, where
$\lambda:=(\lambda_1\ge\lambda_2\ge\cdots\ge\lambda_k\ge 0)$ and $\mu=(\mu_1\ge\mu_2\ge\cdots\ge\mu_l\ge 0)$ are partitions. More specifically, if we let $H(n):=\bQ^n$ denote the standard representation of $\GL_n(\bQ)$ and $H^\vee(n)$ its dual and $\mu$ and $\lambda$ are partitions
of $p$ and $q$ respectively, we define
\[ V_{\lambda,\mu}(n):=(S^{\lambda}\otimes S^{\mu})\otimes_{\Sigma_p\times\Sigma_q}(H(n)^{\otimes p}\otimes H^\vee(n)^{\otimes q}). \]
We define a map
\[ V_{\lambda,\mu}(n+1)\to V_{\lambda,\mu}(n), \]
which is induced by the standard projection $H(n+1)=\bQ^{n+1}\to\bQ^n=H(n)$ and the dual of the standard inclusion $H(n)\hookrightarrow H(n+1)$.
This map is $\GL_n(\bQ)$-equivariant, where the source is considered a $\GL_n(\bQ)$-representation via the standard inclusion $\GL_n(\bQ)\subseteq\GL_{n+1}(\bQ)$.
We define
\[ V_{\lambda,\mu}:=\lim (\cdots \to V_{\lambda,\mu}(n+1)\to V_{\lambda,\mu}(n)\to\cdots \to V_{\lambda,\mu}(0)). \]

Let us call a \emph{non-unital prop} an object defined like a prop, except that we drop the requirement of having units in arity $(1,1)$.
\begin{proposition}
  For each $p,q\ge 0$ and each bipartition $(\lambda,\mu)$ with $|\lambda|=p$ and $|\mu|=q$, the multiplicity of $V_{\lambda,\mu}$ in the decomposition of $H^*_A(\IA_\infty,\bQ)$ into irreducibles of $\GL_\infty(\bQ)$ agrees with the multiplicity
  of $S^\lambda\otimes S^\mu$ in the decomposition of $\cH'(q,p)$, where $\cH'$ is the maximal non-unital sub-prop of $\cH$.
\end{proposition}

If we let $\cP'\subset\cP$ be the symmetric sub-bimodule of $\cP$, generated by partitions with no labeled parts of size one, we have $\cH'\cong\omega(\cP')$.
Defining a symmetric bimodule $\cQ'$ by
\begin{align*}
  \cQ'(q,p) & =
  \begin{cases}
    \triv_p & \text{ if } q=0,\ p\ge 1, \\
    \triv_p & \text{ if } q=1,\ p\ge 2, \\
    0       & \text{ otherwise},
  \end{cases}
  & \chtwo(\cQ') = \sum_{p\ge 1} h_p(y) + \sum_{p\ge 2} h_p(y)h_1(x).
\end{align*}
then $\cP'$ is a regraded version of $\Sat(\cQ')$, so again by Corollary~\ref{cor:bichar-of-sat} (where $\omega$ is the involution from \eqref{eq:omega} and $\Psi$ is the map from \eqref{eq:map-Psi}), we have:
\begin{proposition}\label{prop:char-of-H-prime}
  The character of $\cH'$ is given by the involution of the regrading of the saturation of $\cQ'$:
  \begin{equation}
    \chtwo(\cH')=\omega\Psi\bigl(\chtwo(\Sat(\cQ'))\bigr) = \omega\Psi\Biggl(\sum_{n \geq 1} h_n(x) \bcirc \Bigl( \sum_{p\ge 1} h_p(y) + \sum_{p\ge 2} h_p(y)h_1(x) \Bigr) \Biggr).
  \end{equation}
\end{proposition}
The result of the computation of $\chtwo(\cH')$ in low degree is included in Appendix~\ref{sec:appendixA}.

\appendix
\section{Code and computations}\label{sec:appendixA}

We have implemented most of the operations described in this paper with Mathematica (using character tables for the symmetric groups produced by GAP) and published the code~\cite{ILCode}.

\paragraph{Stable twisted cohomology of $\Aut(F_n)$}

The characters of the symmetric bimodules $\cH^d = \{ H^d(\Aut(F_\infty), B_{\infty}(q,p)) \}_{p,q \in \bN}$ decompose as follows in terms of irreducible representations of symmetric groups.
Since $\cH^d$ is infinite dimensional for all $d \geq 0$, we must truncate their characters; we arbitrarily decided to truncate up to arity $p \leq 4$.
Let us write $S_{\lambda,\mu}$ for the irreducible representation $S_\lambda \otimes S_\mu$ of $\SymG_q \times \SymG_p$ corresponding to the partitions $\lambda$ and $\mu$ of $q$ and $p$ respectively.

\begingroup
\small
\begin{equation}
  \cH^0 =
  \parbox[t]{.75\textwidth}{\raggedright$
    S_{4,4}\oplus S_{31,31}\oplus S_{3,3}\oplus S_{21^2,21^2}\oplus S_{21,21}\oplus S_{2^2,2^2}\oplus S_{2,2}\oplus S_{1^4,1^4}\oplus S_{1^3,1^3}\oplus S_{1^2,1^2}\oplus S_{1,1}\oplus S_{\emptyset ,\emptyset }
    \oplus \dots
  $}
\end{equation}
\endgroup
Note that $\cH^0 = \bigoplus_{\lambda} S_\lambda \otimes S_\lambda$ has character $\prod_{i,j} (1-x_i y_j)^{-1}$.
It would be interesting to find such a ``compact'' expression for the characters of the higher degree terms.
\begingroup
\allowdisplaybreaks
\small
\begin{align}
  \cH^1 & =
  \parbox[t]{.75\textwidth}{\raggedright$
    S_{3,4}\oplus S_{3,31}^{\oplus 2}\oplus S_{3,21^2}\oplus S_{21,31}^{\oplus 2}\oplus S_{21,21^2}^{\oplus 3}\oplus S_{21,2^2}^{\oplus 2}\oplus S_{21,1^4}\oplus S_{2,3}\oplus S_{2,21}^{\oplus 2}\oplus S_{2,1^3}\oplus S_{1^3,21^2}^{\oplus 2}\oplus S_{1^3,2^2}\oplus S_{1^3,1^4}^{\oplus 2}\oplus S_{1^2,21}^{\oplus 2}\oplus S_{1^2,1^3}^{\oplus 2}\oplus S_{1,2}\oplus S_{1,1^2}^{\oplus 2}\oplus S_{\emptyset ,1}
    \oplus \dots
  $} \\
  \cH^2 & =
  \parbox[t]{.75\textwidth}{\raggedright$
    S_{2,31}^{\oplus 3}\oplus S_{2,21^2}^{\oplus 6}\oplus S_{2,2^2}\oplus S_{2,1^4}^{\oplus 2}\oplus S_{1^2,31}\oplus S_{1^2,21^2}^{\oplus 5}\oplus S_{1^2,2^2}^{\oplus 4}\oplus S_{1^2,1^4}^{\oplus 5}\oplus S_{1,21}^{\oplus 3}\oplus S_{1,1^3}^{\oplus 4}\oplus S_{\emptyset ,1^2}^{\oplus 2}
    \oplus \dots
  $} \\
  \cH^3 & =
  \parbox[t]{.75\textwidth}{\raggedright$
    S_{1,31}\oplus S_{1,21^2}^{\oplus 7}\oplus S_{1,2^2}^{\oplus 3}\oplus S_{1,1^4}^{\oplus 7}\oplus S_{\emptyset ,21}\oplus S_{\emptyset ,1^3}^{\oplus 3}
    \oplus \dots
  $} \\
  \cH^4 & =
  \parbox[t]{.75\textwidth}{\raggedright$
    S_{\emptyset ,21^2}^{\oplus 2}\oplus S_{\emptyset ,2^2}^{\oplus 2}\oplus S_{\emptyset ,1^4}^{\oplus 5}
    \oplus \dots
  $}
\end{align}
\endgroup

\paragraph{Albanese cohomology of $\IA_n$}

We obtain the following decompositions for the Albanese cohomology $H^d_A(\IA_\infty,\bQ)$.
We stop at $d \leq 5$ for presentation reasons; Table~\ref{tab:irr-comp-hdaIA} gives the number of irreducible sub-representations of the Albanese cohomology for $d \leq 10$.
Note that our computations match with those of~\cite{Kawazumi2005-MagnusExpansions,Pettet05,Katada2022} for $d \leq 3$.

\begin{table}[htbp]
  \centering
  \begin{tabular}{r|cccccccccc}
    $d$        & 1 & 2 & 3  & 4   & 5   & 6            & 7             & 8              & 9              & 10 \\
    \#irr      & 2 & 6 & 21 & 69  & 219 & 663          & \text{1,915}  & \text{5,182}   & \text{13,330}  & \text{32,876} \\
    $\sum$mult & 2 & 8 & 34 & 152 & 720 & \text{3,634} & \text{19,266} & \text{107,018} & \text{619,606} & \text{3,727,224} \\
  \end{tabular}
  \caption{The number of irreducible sub-representations of $H^d_A(\IA_\infty, \bQ)$ and the sum of their multiplicities.}
  \label{tab:irr-comp-hdaIA}
\end{table}

\begingroup
\allowdisplaybreaks
\small
\begin{align}
  H^1_A(\IA_\infty, \bQ) & =
  \parbox[t]{.75\textwidth}{\raggedright$
    V_{1,1^2}\oplus V_{\emptyset ,1}
  $} \\
  H^2_A(\IA_\infty, \bQ) & =
  \parbox[t]{.75\textwidth}{\raggedright$
    V_{2,21^2}\oplus V_{1^2,2^2}\oplus V_{1^2,1^4}\oplus V_{1,21}\oplus V_{1,1^3}^{\oplus 2}\oplus V_{\emptyset ,1^2}^{\oplus 2}
  $} \\
  H^3_A(\IA_\infty, \bQ) & =
  \parbox[t]{.75\textwidth}{\raggedright$
    V_{3,31^3}\oplus V_{3,2^3}\oplus V_{21,321}\oplus V_{21,2^21^2}\oplus V_{21,21^4}\oplus V_{2,31^2}\oplus V_{2,2^21}^{\oplus 2}\oplus V_{2,21^3}^{\oplus 2}\oplus V_{2,1^5}\oplus V_{1^3,3^2}\oplus V_{1^3,2^21^2}\oplus V_{1^3,1^6}\oplus V_{1^2,32}\oplus V_{1^2,2^21}^{\oplus 2}\oplus V_{1^2,21^3}^{\oplus 2}\oplus V_{1^2,1^5}^{\oplus 2}\oplus V_{1,21^2}^{\oplus 3}\oplus V_{1,2^2}^{\oplus 2}\oplus V_{1,1^4}^{\oplus 4}\oplus V_{\emptyset ,21}\oplus V_{\emptyset ,1^3}^{\oplus 3}
  $} \\
  H^4_A(\IA_\infty, \bQ) & =
  \parbox[t]{.75\textwidth}{\raggedright$
    V_{4,41^4}\oplus V_{4,32^21}\oplus V_{31,421^2}\oplus V_{31,3^22}\oplus V_{31,32^21}\oplus V_{31,321^3}\oplus V_{31,31^5}\oplus V_{31,2^31^2}\oplus V_{3,41^3}\oplus V_{3,32^2}^{\oplus 2}\oplus V_{3,321^2}^{\oplus 2}\oplus V_{3,31^4}^{\oplus 2}\oplus V_{3,2^31}^{\oplus 2}\oplus V_{3,2^21^3}\oplus V_{3,21^5}\oplus V_{21^2,431}\oplus V_{21^2,3^21^2}\oplus V_{21^2,32^21}\oplus V_{21^2,321^3}\oplus V_{21^2,2^31^2}\oplus V_{21^2,2^21^4}\oplus V_{21^2,21^6}\oplus V_{21,421}\oplus V_{21,3^21}^{\oplus 2}\oplus V_{21,32^2}^{\oplus 2}\oplus V_{21,321^2}^{\oplus 4}\oplus V_{21,31^4}^{\oplus 2}\oplus V_{21,2^31}^{\oplus 3}\oplus V_{21,2^21^3}^{\oplus 5}\oplus V_{21,21^5}^{\oplus 3}\oplus V_{21,1^7}\oplus V_{2^2,42^2}\oplus V_{2^2,3^21^2}\oplus V_{2^2,321^3}\oplus V_{2^2,2^21^4}\oplus V_{2^2,2^4}\oplus V_{2,321}^{\oplus 3}\oplus V_{2,31^3}^{\oplus 3}\oplus V_{2,2^21^2}^{\oplus 5}\oplus V_{2,21^4}^{\oplus 6}\oplus V_{2,2^3}^{\oplus 4}\oplus V_{2,1^6}^{\oplus 2}\oplus V_{1^4,4^2}\oplus V_{1^4,3^21^2}\oplus V_{1^4,2^21^4}\oplus V_{1^4,2^4}\oplus V_{1^4,1^8}\oplus V_{1^3,43}\oplus V_{1^3,3^21}^{\oplus 2}\oplus V_{1^3,321^2}^{\oplus 2}\oplus V_{1^3,2^31}^{\oplus 2}\oplus V_{1^3,2^21^3}^{\oplus 3}\oplus V_{1^3,21^5}^{\oplus 2}\oplus V_{1^3,1^7}^{\oplus 2}\oplus V_{1^2,321}^{\oplus 3}\oplus V_{1^2,31^3}\oplus V_{1^2,3^2}^{\oplus 2}\oplus V_{1^2,2^21^2}^{\oplus 8}\oplus V_{1^2,21^4}^{\oplus 5}\oplus V_{1^2,2^3}\oplus V_{1^2,1^6}^{\oplus 5}\oplus V_{1,32}\oplus V_{1,31^2}\oplus V_{1,2^21}^{\oplus 6}\oplus V_{1,21^3}^{\oplus 7}\oplus V_{1,1^5}^{\oplus 7}\oplus V_{\emptyset ,21^2}^{\oplus 2}\oplus V_{\emptyset ,2^2}^{\oplus 2}\oplus V_{\emptyset ,1^4}^{\oplus 5}
  $} \\
  H^5_A(\IA_\infty, \bQ) & =
  \parbox[t]{.75\textwidth}{\raggedright$
    V_{5,51^5}\oplus V_{5,42^21^2}\oplus V_{5,3^22^2}\oplus V_{41,521^3}\oplus V_{41,4321}\oplus V_{41,42^3}\oplus V_{41,42^21^2}\oplus V_{41,421^4}\oplus V_{41,41^6}\oplus V_{41,3^31}\oplus V_{41,3^221^2}\oplus V_{41,32^31}\oplus V_{41,32^21^3}\oplus V_{4,51^4}\oplus V_{4,42^21}^{\oplus 2}\oplus V_{4,421^3}^{\oplus 2}\oplus V_{4,41^5}^{\oplus 2}\oplus V_{4,3^221}^{\oplus 2}\oplus V_{4,32^3}^{\oplus 2}\oplus V_{4,32^21^2}^{\oplus 3}\oplus V_{4,321^4}\oplus V_{4,31^6}\oplus V_{4,3^3}\oplus V_{4,2^31^3}\oplus V_{32,52^21}\oplus V_{32,43^2}\oplus V_{32,4321}\oplus V_{32,431^3}\oplus V_{32,42^21^2}\oplus V_{32,421^4}\oplus V_{32,3^22^2}\oplus V_{32,3^221^2}\oplus V_{32,3^21^4}\oplus V_{32,32^31}\oplus V_{32,32^21^3}\oplus V_{32,321^5}\oplus V_{32,2^41^2}\oplus V_{31^2,531^2}\oplus V_{31^2,4^22}\oplus V_{31^2,4321}\oplus V_{31^2,431^3}\oplus V_{31^2,42^3}\oplus V_{31^2,42^21^2}\oplus V_{31^2,421^4}\oplus V_{31^2,3^31}\oplus V_{31^2,3^221^2}^{\oplus 2}\oplus V_{31^2,32^31}\oplus V_{31^2,32^21^3}^{\oplus 2}\oplus V_{31^2,321^5}\oplus V_{31^2,31^7}\oplus V_{31^2,2^31^4}\oplus V_{31^2,2^5}\oplus V_{31,521^2}\oplus V_{31,432}^{\oplus 2}\oplus V_{31,431^2}^{\oplus 2}\oplus V_{31,42^21}^{\oplus 4}\oplus V_{31,421^3}^{\oplus 4}\oplus V_{31,41^5}^{\oplus 2}\oplus V_{31,3^221}^{\oplus 5}\oplus V_{31,3^21^3}^{\oplus 3}\oplus V_{31,32^3}^{\oplus 3}\oplus V_{31,32^21^2}^{\oplus 8}\oplus V_{31,321^4}^{\oplus 6}\oplus V_{31,31^6}^{\oplus 3}\oplus V_{31,3^3}^{\oplus 2}\oplus V_{31,2^41}^{\oplus 3}\oplus V_{31,2^31^3}^{\oplus 4}\oplus V_{31,2^21^5}^{\oplus 2}\oplus V_{31,21^7}\oplus V_{3,42^2}\oplus V_{3,421^2}^{\oplus 3}\oplus V_{3,41^4}^{\oplus 3}\oplus V_{3,3^22}^{\oplus 4}\oplus V_{3,3^21^2}\oplus V_{3,32^21}^{\oplus 9}\oplus V_{3,321^3}^{\oplus 7}\oplus V_{3,31^5}^{\oplus 6}\oplus V_{3,2^31^2}^{\oplus 7}\oplus V_{3,2^21^4}^{\oplus 4}\oplus V_{3,21^6}^{\oplus 3}\oplus V_{3,2^4}\oplus V_{2^21,532}\oplus V_{2^21,4^21^2}\oplus V_{2^21,4321}\oplus V_{2^21,431^3}\oplus V_{2^21,42^21^2}\oplus V_{2^21,3^22^2}\oplus V_{2^21,3^221^2}\oplus V_{2^21,3^21^4}^{\oplus 2}\oplus V_{2^21,32^31}\oplus V_{2^21,32^21^3}\oplus V_{2^21,321^5}\oplus V_{2^21,2^41^2}\oplus V_{2^21,2^31^4}\oplus V_{2^21,2^21^6}\oplus V_{21^3,541}\oplus V_{21^3,4^21^2}\oplus V_{21^3,4321}\oplus V_{21^3,431^3}\oplus V_{21^3,3^22^2}\oplus V_{21^3,3^221^2}\oplus V_{21^3,3^21^4}\oplus V_{21^3,32^31}\oplus V_{21^3,32^21^3}\oplus V_{21^3,321^5}\oplus V_{21^3,2^41^2}\oplus V_{21^3,2^31^4}\oplus V_{21^3,2^21^6}\oplus V_{21^3,21^8}\oplus V_{21^2,531}\oplus V_{21^2,4^21}^{\oplus 2}\oplus V_{21^2,432}^{\oplus 2}\oplus V_{21^2,431^2}^{\oplus 4}\oplus V_{21^2,42^21}^{\oplus 2}\oplus V_{21^2,421^3}^{\oplus 2}\oplus V_{21^2,3^221}^{\oplus 5}\oplus V_{21^2,3^21^3}^{\oplus 6}\oplus V_{21^2,32^3}^{\oplus 3}\oplus V_{21^2,32^21^2}^{\oplus 7}\oplus V_{21^2,321^4}^{\oplus 6}\oplus V_{21^2,31^6}^{\oplus 2}\oplus V_{21^2,2^41}^{\oplus 4}\oplus V_{21^2,2^31^3}^{\oplus 6}\oplus V_{21^2,2^21^5}^{\oplus 6}\oplus V_{21^2,21^7}^{\oplus 3}\oplus V_{21^2,1^9}\oplus V_{21,431}^{\oplus 3}\oplus V_{21,42^2}^{\oplus 3}\oplus V_{21,421^2}^{\oplus 4}\oplus V_{21,41^4}\oplus V_{21,3^22}^{\oplus 5}\oplus V_{21,3^21^2}^{\oplus 9}\oplus V_{21,32^21}^{\oplus 12}\oplus V_{21,321^3}^{\oplus 16}\oplus V_{21,31^5}^{\oplus 7}\oplus V_{21,2^31^2}^{\oplus 14}\oplus V_{21,2^21^4}^{\oplus 15}\oplus V_{21,21^6}^{\oplus 9}\oplus V_{21,2^4}^{\oplus 6}\oplus V_{21,1^8}^{\oplus 3}\oplus V_{2^2,52^2}\oplus V_{2^2,432}^{\oplus 2}\oplus V_{2^2,431^2}^{\oplus 2}\oplus V_{2^2,42^21}^{\oplus 2}\oplus V_{2^2,421^3}^{\oplus 2}\oplus V_{2^2,3^221}^{\oplus 3}\oplus V_{2^2,3^21^3}^{\oplus 4}\oplus V_{2^2,32^3}^{\oplus 2}\oplus V_{2^2,32^21^2}^{\oplus 4}\oplus V_{2^2,321^4}^{\oplus 5}\oplus V_{2^2,31^6}\oplus V_{2^2,2^41}^{\oplus 3}\oplus V_{2^2,2^31^3}^{\oplus 3}\oplus V_{2^2,2^21^5}^{\oplus 3}\oplus V_{2^2,21^7}\oplus V_{2,421}\oplus V_{2,41^3}\oplus V_{2,3^21}^{\oplus 3}\oplus V_{2,32^2}^{\oplus 7}\oplus V_{2,321^2}^{\oplus 10}\oplus V_{2,31^4}^{\oplus 8}\oplus V_{2,2^31}^{\oplus 11}\oplus V_{2,2^21^3}^{\oplus 15}\oplus V_{2,21^5}^{\oplus 12}\oplus V_{2,1^7}^{\oplus 5}\oplus V_{1^5,5^2}\oplus V_{1^5,4^21^2}\oplus V_{1^5,3^22^2}\oplus V_{1^5,3^21^4}\oplus V_{1^5,2^41^2}\oplus V_{1^5,2^21^6}\oplus V_{1^5,1^{10}}\oplus V_{1^4,54}\oplus V_{1^4,4^21}^{\oplus 2}\oplus V_{1^4,431^2}^{\oplus 2}\oplus V_{1^4,3^221}^{\oplus 2}\oplus V_{1^4,3^21^3}^{\oplus 3}\oplus V_{1^4,32^3}^{\oplus 2}\oplus V_{1^4,32^21^2}\oplus V_{1^4,321^4}^{\oplus 2}\oplus V_{1^4,2^41}^{\oplus 2}\oplus V_{1^4,2^31^3}^{\oplus 3}\oplus V_{1^4,2^21^5}^{\oplus 3}\oplus V_{1^4,21^7}^{\oplus 2}\oplus V_{1^4,1^9}^{\oplus 2}\oplus V_{1^3,431}^{\oplus 3}\oplus V_{1^3,421^2}\oplus V_{1^3,4^2}^{\oplus 2}\oplus V_{1^3,3^22}\oplus V_{1^3,3^21^2}^{\oplus 8}\oplus V_{1^3,32^21}^{\oplus 5}\oplus V_{1^3,321^3}^{\oplus 7}\oplus V_{1^3,31^5}\oplus V_{1^3,2^31^2}^{\oplus 7}\oplus V_{1^3,2^21^4}^{\oplus 11}\oplus V_{1^3,21^6}^{\oplus 6}\oplus V_{1^3,2^4}^{\oplus 5}\oplus V_{1^3,1^8}^{\oplus 5}\oplus V_{1^2,43}\oplus V_{1^2,421}\oplus V_{1^2,3^21}^{\oplus 6}\oplus V_{1^2,32^2}^{\oplus 3}\oplus V_{1^2,321^2}^{\oplus 11}\oplus V_{1^2,31^4}^{\oplus 4}\oplus V_{1^2,2^31}^{\oplus 11}\oplus V_{1^2,2^21^3}^{\oplus 18}\oplus V_{1^2,21^5}^{\oplus 13}\oplus V_{1^2,1^7}^{\oplus 9}\oplus V_{1,321}^{\oplus 5}\oplus V_{1,31^3}^{\oplus 3}\oplus V_{1,3^2}^{\oplus 2}\oplus V_{1,2^21^2}^{\oplus 15}\oplus V_{1,21^4}^{\oplus 14}\oplus V_{1,2^3}^{\oplus 5}\oplus V_{1,1^6}^{\oplus 12}\oplus V_{\emptyset ,32}\oplus V_{\emptyset ,2^21}^{\oplus 4}\oplus V_{\emptyset ,21^3}^{\oplus 5}\oplus V_{\emptyset ,1^5}^{\oplus 7}
  $}
\end{align}
\endgroup

\printbibliography
\end{document}